\documentclass[a4paper, 11pt, twoside, final]{scrartcl}

\usepackage{preamble}
\usepackage{command}
\usepackage{hyphenation}
\usepackage{hyperref}
\usepackage{cite}

\usepackage[draft, todonotes={textsize=scriptsize}]{changes}

\setuptodonotes{color=Crimson, backgroundcolor=Crimson!50,
  bordercolor=Crimson, tickmarkheight=10pt}

\definechangesauthor[name=Robert Beinert, color=cyan]{RB}
\definechangesauthor[name=Marzieh Hasannasab, color=magenta]{MH}

\AUTHOR{Robert \pers{Beinert}\Inst{1} and Marzieh
  \pers{Hasannasab}\Inst{1}}

\SHORTAUTHOR{R.~Beinert and
  M.~Hasannasab}

\TITLE{Phase Retrieval and System Identification in\\ Dynamical Sampling
  via Prony's Method}

\SHORTTITLE{Phase Retrieval \& System Identification in Dynamical
  Sampling}

\INSTITUTE{\Inst{1}\parbox[t]{0.95\linewidth}{Institut für Mathematik\\
    Technische Universität Berlin\\
    Straße des 17. Juni 136\\
    10623 Berlin, Germany } }

\CORRESPONDENCE{R. \pers{Beinert}: \email{beinert@math.tu.berlin.de}\\
  M. \pers{Hasannasab}: \email{hasannas@math.tu-berlin.de}}

\ABSTRACT{Phase retrieval in dynamical sampling is a novel research
  direction, where an unknown signal has to be recovered from the
  phaseless measurements with respect to a dynamical frame, \ie\ a
  sequence of sampling vectors constructed by the repeated action of
  an operator.  The loss of the phase here turns the well-posed
  dynamical sampling into a severe ill-posed inverse problem.  In the
  existing literature, the involved operator is usually completely
  known.  In this paper, we combine phase retrieval in dynamical
  sampling with the identification of the system.  For instance, if the
  dynamical frame is based on a repeated convolution, then we want to
  recover the unknown convolution kernel in advance.  Using Prony's
  method, we establish several recovery guarantees for signal and
  system, whose proofs are constructive and yield analytic recovery
  methods.  The required assumptions are satisfied by almost all
  signals, operators, and sampling vectors.  Moreover, these
  guarantees not only hold for the finite-dimensional setting but also
  carry over to infinite-dimensional spaces.  Studying the sensitivity
  of the analytic recovery procedures, we also establish error bounds
  for the applied approximate Prony method with respect to complex exponential
  sums.}


\AMSCLASS{42A05, 94A12, 15A29, 94A20 }

\begin{document}
\thispagestyle{plain}
\TitleHeader

\section{Introduction}
\label{sec:introduction}

Phase retrieval is an ill-posed inverse problem consisting in the
recovery of signals or images from phaseless measurements like the
magnitude of the Fourier transform or the absolute values of inner
products with respect to given sampling vectors. Phaseless
reconstructions appear naturally in many applications like X-ray
crystallography \cite{Hau91,KH91,millane1990phase}, astronomy
\cite{BS79,fienup1987phase}, laser optics \cite{SSD+06,SST04} and
audio processing \cite{dell2000, flan66, laro99}. The mathematical
analysis of this ill-posed problem has been studied intensively during
the last decades, see for instance
\cite{ADGY19,BP15,BP20,BBE17,BS79,grohs2019mathematics,HHLO83,KK14,KST95,SECC+15,
  alaifari2021stability} and references therein.

In this paper, we consider phase retrieval in the context of dynamical
sampling.  Dynamical sampling is a novel research direction motivated
by the work of Vetterli \etal\ \cite{LV09,RCLV11} and was introduced
in \cite{ACMT17,AHP19,AK16,AP17}. The topic instantly attracted
attention in the scientific community, see for instance
\cite{cabrelli2020dynamical, christensen2020frame,
  christensen2019frame, martin2021continuous, philipp2017bessel,
  aceska2017scalability, ulanovskii2021reconstruction,
  martin2021continuous, AGH+20, tang2017} for further
studies. Formulated in the setting of finite-dimensional spaces, the
main question in dynamical sampling is to find conditions on the
system $\Mat A \in \BC^{d \times d}$ and the sampling vectors
$\{\Vek\phi_i\}_{i=0}^{J-1} \subset \BC^d$ such that each signal
$\Vek x\in \BC^d$ can be stably recovered from the spatiotemporal
samples
\begin{equation*}
  \bigl\{ \langle \Vek x , \Mat A^\ell\Vek\phi_i\rangle
  \bigr\}_{\ell,i=0}^{L-1,J-1}
\end{equation*}
or such that $\{\Mat A^\ell \Vek \phi_i\}_{\ell,i=0}^{L-1,J-1}$ forms
a frame for some $L, J \in \BN$ .  Note that many structured
measurements like the discrete Gabor transform may be interpreted as
dynamical samples.  For the Gabor transform, $\Mat A$ would be a
diagonal matrix corresponding to the modulation operator, and
$\Vek \phi_i$ would be shifts of a window function.  We refer to
\cite{ACMT17,AK16} for motivations about this particular question.

Different from the classical finite-dimensional dynamical sampling, we
consider the phaseless measurements
\begin{equation*}
  \bigl\{
  \absn{\iProdn{\Vek x}{\Mat A^\ell \Vek \phi_i}}^2
  \bigr\}_{\ell,i=0}^{L-1,J-1}
\end{equation*}
for some $L, J\in\BN$. The main question is again: under which
conditions on $\Mat A$ and $\Vek \phi_i$ can $\Vek x$ be recovered
from the given measurements.  Due to the loss of the phase, this
problem becomes far more challenging since the recovery is now
severely ill posed in advance.

\paragraph{Relation to existing literature}

Phase retrieval in dynamical sampling has already been studied.  In
\cite{aldroubi2020phaseless,aldroubi2017phase}, the authors pose
conditions on the operator $\Mat A$ defined on a real Hilbert space
and on the sampling vectors $\Vek\phi_i$ to ensure that the dynamical
phase retrieval problem has a unique solution.  The main strategy is
here to ensure that the sequence
$\{\Mat A^\ell \Vek\phi_i\}_{\ell=i=0}^{L_i-1,J-1}$ has the
complementary property meaning that each subset or its complement
spans the entire space.  The restriction to the real-valued problem is
crucial since the complementary property is not sufficient to allow
phase retrieval in the complex case.  Further, the results are of a
theoretical nature, and the question how to recover the signal
numerical remains open.  

An approach for a numerical recovery procedure based on polarization
identities has been considered in \cite{BH21}, where the measurement
vectors $\Vek \phi_i$ have been designed to allow phase
retrieval.  The key idea has been to consider interfering measurement
vectors that allow the recovery of the missing phase by polarization
such that we obtain a classical dynamical sampling problem, which can
be solved in a second step.  The presented reconstruction technique
works for almost all real or complex signals.

\paragraph{Contributions}

Besides the recovery of the real or complex signal $\Vek x$, we want
to recover the unknown operator $\Mat A$ from a certain class in
advance.  For instance, if the operator
$\Mat A \coloneqq \Circ \Vek a$ corresponds to the convolution with
$\Vek a$, we want to recover the signal $\Vek a$ or the spectrum
$\hat{\Vek a}$ of $\Mat A$, where $\hat\cdot$ denotes the discrete
Fourier transform.  The theoretical requirements to allow phase
retrieval besides system identification is our main contribution and
focus of this paper.  The combination of phase retrieval, dynamical
sampling, and system identification is to our knowledge a
new research topic.  Our work horse to establish the recovery
guarantees for phase and system is Prony's method, which allow us to
recover the wanted entities from the given measurements.  As a result,
all our proofs contain analytic recovery methods.  The required
assumptions are satisfied by almost all signals, spectra, and sampling
vectors.  Using several sampling vectors, phase retrieval and system
identification is possible from only linearly many samples.  The basic
idea here generalizes to the infinite-dimensional setting.  Moreover,
we study the sensitivity of the applied Prony method resulting in
error bounds that are interesting by their own outside the context of
dynamical sampling.  On this basis, we moreover study the sensitivity
of the proposed analytic recovery procedures.
 
\paragraph{Roadmap}

This paper is organized as follows. In Section
\ref{sec:preliminary-notes}, we introduce the required notations.  In
Section \ref{sec:appr-prony-meth}, we recall Prony's method, and we
explain how this method enables us to recover the missing
information. In Section \ref{sec:excl-phase-retr}, we provide
conditions to retrieve an unknown signal when the underlying dynamical
frame is known. Section \ref{sec:excl-syst-ident} is devoted to the
system identification in case that the signal $\Vek x$ is already
known.  In Section~\ref{sec:simult-phase-system}, we suppose that both
the signal and the spectrum of $\Mat A$ are unknown.  In particular,
we establish recovery guarantees when the operator $\Mat A$
corresponds to a convolution with a low-pass filter as kernel.  In
Section \ref{sec:multi-samp-vec}, we consider multiple sampling
vectors, which finally allow us to recover both -- signal and
operator. In Section \ref{sec:infinite-dim}, we adapt our results to
the infinite-dimensional setting.  The sensitivity of the analytic
reconstructions is investigated in Section \ref{sec:sens-analysis}. In
Section \ref{sec:numerical-examples}, we provide numerical examples to
accompany our theoretical results. Section \ref{sec:conclusion}
concludes the paper with a number of final remarks.

\section{Preliminary notes}
\label{sec:preliminary-notes}

In this section, we introduce the notations and definitions that are
needed throughout this paper.  All finite-dimensional vectors and
matrices are stated in bold print.  The zero matrix of dimension
$L\times K$ is denoted by $\Mat 0 \coloneqq \Mat 0_{L,K}$ and the
($d \times d$)-dimensional identity by $\Mat I \coloneqq \Mat I_d$.
If the dimension is clear within the context, we usually skip the
indices.

A matrix $\Mat A \in \BC^{d\times d}$ is called diagonalizable if
there exist an invertible matrix $\Mat S$, whose columns consists of
eigenvectors of $\Mat A$, and a diagonal matrix $\Mat \Lambda$ with
the eigenvalues of $\Mat A$ on its diagonal, such that
$\Mat A = \Mat S \Mat \Lambda \Mat S^{-1}$.  Throughout the paper, we
always use this eigenvalue decomposition, where $\Mat S$ does not have
to be orthogonal implying that the columns of $\Mat S$ only form a
(maybe non-orthogonal) basis.  Further, if the eigenvalues are
pairwise distinct, we say that a given vector $\Vek\phi\in\BC^d$
\emph{depends on all eigenspaces} of $\Mat A$ if $\Mat S^{-1}\Vek\phi$
does not vanish anywhere, \ie\ if all coordinate to the basis in
$\Mat S$ are non-zero.  Note that in this case $\Mat S$ is unique up
to permutation and global phase of the columns.

For $\Vek a \in \BC^d$, we denote by $\Circ(\Vek a)$ the circulant matrix
whose first column is $\Vek a$. Note that the
multiplication with $\Circ(\Vek a)$ results in the convolution with
$\Vek a$, \ie\ $\Circ(\Vek a) \, \Vek x = \Vek a * \Vek x$.  All
circulant matrices are diagonalizable with respect to the discrete
Fourier transform.  More precisely, we have
$\Circ(\Vek a)= \nicefrac{1}{d} \, \Mat F \diag(\hat{\Vek a}) \, \Mat
F^{-1}$, where
$\Mat F = (\e^{\nicefrac{-2\uppi \I j k}{d}})_{j,k=0}^{d-1}$ denotes
the Fourier matrix and $\hat{\Vek a} \coloneqq \Mat F \Vek a$ the
discrete Fourier transform.

Given a vector $\Vek\beta \in \BC^{K}$ and $L\in \BN$, we define the
rectangular Vandermonde matrix $\Mat V_L\in \BC^{L\times K}$ by
\begin{equation*}
  \Mat V_L \coloneqq \Mat V_L(\Vek \beta) \coloneqq
  (\beta_k^\ell)_{\ell,k = 0}^{L-1,K-1}.
  \addmathskip
\end{equation*}
For $L=K$, we drop the subscript and denote the Vandermonde
matrix by $\Mat V$ or $\Mat V(\Vek\beta)$.

Recall that the finite-dimensional $p$-norm is defined as
\begin{equation*}
  \| \Vek x\|_p
  = \biggl(\sum_{k=0}^{d-1} | x_k |^p \biggr)^{\nicefrac1p}
  \qquad \text{for}\qquad
  \Vek x\in \BC^d
  \quad\text{and}\quad
  p \in [1,\infty).
\end{equation*}
Moreover, the maximum norm is defined by
$ \| \Vek x\|_{\infty} = \max_k | x_k |$.  Against this background and
for notational convenience, we define the minimum norm
$ \| \Vek x \|_{-\infty} =\min_k |x_k |$ although this expression is
clearly no norm.

The non-zero complex numbers are denote by $\BC_*$.  Without loss of
generality, we always choose the phase $\arg(\cdot)$ of a complex
number within the interval $[-\uppi, \uppi)$.  Especially for
calculations with phases, we denote by $\cdot \Mod 2\uppi$ the
remainder within $[-\uppi, \uppi)$, \ie\ we add or subtract a multiple
of $2\uppi$ to obtain an number in the considered interval.

For a given vector $\Vek x=(x_0,\dots,x_{d-1})$, we call the set of
relative phases $\arg(x_j \bar x_k)$ the \textit{winding direction} of
$\Vek x$.  Figuratively, the winding direction describes how the phase
is changing by traveling through the components of $\Vek x$.  We say
that a vector $\Vek x$ can be uniquely recovered up to the winding
direction if the relatives phases are only reconstructable up to a
global sign.  If $x_0$ is real, a vector
with the opposite winding direction can be computed by conjugating all
components of $\Vek x$, \ie\ changing the sign of all relative phases.

Finally, we denote by $ \#[ \cdot ]$ the cardinality of a set.

\section{The approximate Prony method}
\label{sec:appr-prony-meth}

In a nutshell, Prony's method \cite{Pro95} allows us to recover the
non-zero coefficients $\eta_k \in \BC_*$ and the
pairwise distinct bases $\beta_k \in \BC_*$ of an
exponential sum
\begin{equation}
  \label{eq:exp-sum}
  f(t) \coloneqq \sum_{k=0}^{K-1} \eta_k \, \beta_k^t
\end{equation}
from the equispaced sampled data $h_\ell \coloneqq f(\ell)$ with
$\ell = 0, \dots, 2K-1$.  The so-called Prony polynomial
$P \colon \BC \to \BC $ is the monic polynomial whose zeros are the
unknown bases, \ie\
$P(z) \coloneqq \sum_{k=0}^K \gamma_k z^k = \prod_{k=0}^{K-1} (z -
\beta_k)$ with $\gamma_K = 1$.  Considering the linear equations
\begin{equation}
  \label{eq:hankel-prony}
  \sum_{k=0}^K \gamma_k \, h_{\ell + k} =
  \sum_{j=0}^{K-1} \eta_j \beta_j^\ell \, P(\beta_j) = 0, \qquad \ell = 0,
  \dots, K-1,
\end{equation}
one may calculate the coefficients $\gamma_k$ of the Prony polynomial
by solving a linear equation system.  Knowing the Prony polynomial, we
may extract the unknown bases $\beta_k$ via its roots.  The
coefficients $\eta_k$ of the exponential sum are determined by an
over-determined linear equation system.  To improve the numerical
performance, the number of measurements may be increased
\cite{BM05,PDV05,PT10}.  On the basis of the rectangular Hankel matrix
\begin{equation}
  \label{eq:hankel}
  \Mat H \coloneqq
  \bigl( h_{\ell + k} \bigr)_{\ell,k = 0}^{L-K-1, K}
  \qquad\text{with}\qquad
  L \ge 2K,
\end{equation}
the coefficients of the Prony polynomial are determined by the
kernel of $\Mat H$.  

\begin{lemma}\label{lem:singular-value-and-roots}
  For the exact samples $h_\ell$ with $\ell = 0, \dots, L-1$, the
  rectangular Hankel matrix \eqref{eq:hankel} is of rank $K$, and the
  following assertions are equivalent:
  \begin{enumerate}[\upshape(i),nosep]
  \item the polynomial $P(z) \coloneqq \sum_{\ell=0}^K \gamma_\ell
    z^\ell$ has the $K$ distinct roots
    $\beta_0, \dots, \beta_{K-1}$,
  \item the vector $\Vek \gamma \coloneqq (\gamma_\ell)_{\ell = 0}^K$
    spans $\ker(\Mat H)$, \ie\ $\Mat H \Vek \gamma
    = \Vek 0$.
  \end{enumerate}
\end{lemma}

\begin{proof}
  With $\Vek \eta \coloneqq (\eta_k)_{k=0}^{K-1}$ and
  $\Vek \beta \coloneqq ( \beta_k)_{k=0}^{K-1}$, we may factorize the
  Hankel matrix \eqref{eq:hankel} into
  \begin{equation*}
    \Mat H = \Vek V_{L-K}(\Vek \beta) \diag(\Vek \eta) \, \Mat
    V_{K+1}^\T(\Vek \beta).
  \end{equation*}
  Since the occurring Vandermonde and diagonal matrices have full
  rank, we have\linebreak $\rank \Mat H = K$ meaning $\dim(\ker(\Mat H)) = 1$.
  Thus, $\Mat H$ possesses the simple singular value zero.
  Considering \eqref{eq:hankel-prony} for $\ell = 0, \dots, L-K-1$,
  we obtain
  \begin{equation*}
    \Mat H \Vek \gamma = \Vek V_{L-K}(\Vek \beta) \, \bigl( \eta_j \,
    P(\beta_j) \bigr)_{j=0}^{K-1}.
  \end{equation*}
  Since the Vandermonde matrix $\Mat V_{L-K}$ has full rank due to
  the assumptions on \eqref{eq:exp-sum}, the equivalence follows
  immediately.  \qed
\end{proof}

\thref{lem:singular-value-and-roots} is the theoretical justification
why Prony's method always yields the parameters of \eqref{eq:exp-sum}
for exact measurements $h_\ell$.  In practice, the measurements
$\tilde h_\ell \coloneqq h_\ell + e_\ell$ are disturbed by some small
error $e_\ell$; so we have only access to the disturbed rectangular
Hankel matrix
\begin{equation}
  \label{eq:dist-hankel}
  \tilde{\Mat H} \coloneqq \Mat H + \Mat E
  = \bigl( h_{\ell + k} + e_{\ell + k} \bigr)_{\ell,k = 0}^{L-K-1, K}
  \qquad\text{with}\qquad
  L \ge 2K,
\end{equation}
where $\Mat E \coloneqq (e_{\ell + k})_{\ell,k=0}^{L-K-1,K}$ is the
rectangular error Hankel matrix.  If $L > 2K$, the kernel of the
perturbed Hankel matrix $\tilde{\Mat H}$ will be trivial almost
surely.  For this reason, Potts \& Tasche \cite{PT10} suppose to
approximate the kernel using the singular value decomposition.  This
approach is supported by the Lidskii--Weyl perturbation theorem for
singular values, see \cite[Prob~III.6.13]{Bha97} or \cite{LM99},
yielding
\begin{equation}
  \label{eq:bound-sing-val}
  \max_{k=0,\dots,K}
  \absn{\sigma_{k}(\tilde{\Mat H}) - \sigma_{k}(\Mat H)}
  \le \pNormn{\tilde{\Mat H} - \Mat H}_2
  \le \pNormn{\Mat E}_2.
\end{equation}
If the non-zero singular values of $\Mat H$ are greater than
$2 \pNormn{\Mat E}$, the singular vector to the smallest singular
value of $\tilde{\Mat H}$ seems to be a valid approximation for
$\Vek \gamma$.  Summarized, we obtain the so-called approximate Prony
method \cite[Alg~3.3]{PT10} here written down for complex exponential
sums.

\begin{algorithm}[Approximate Prony method]
  \label{alg:app-prony}
  \emph{Input}:
  $\tilde{\Vek h} \coloneqq (\tilde h_\ell)_{\ell=0}^{L-1} \in
  \BC^{L}$ with $L > 2K$.
  \begin{enumerate}[(i),nosep]
  \item Compute the right singular vector $\tilde{\Vek \gamma}$ to the smallest
    singular value $\sigma_K$ of $\tilde{\Mat H}$.
  \item Determine the roots
    $\tilde{\Vek \beta} \coloneqq (\tilde{\beta}_k)_{k=0}^{K-1}$ of
    $\tilde{P}(z) = \sum_{k=0}^K \tilde{\gamma}_k z^k$.
  \item Compute the least-squares solution of $\Mat V_{L}(\tilde{\Vek \beta})
    \, \tilde{\Vek \eta }= \tilde{\Vek h}$.
  \end{enumerate}
  \emph{Output}: $\tilde{\Vek \eta} \in \BC^K$, $\tilde{\Vek \beta} \in \BC^K$.
\end{algorithm}

Finally, we would like to note that alternative methods to obtain
unknown bases from the exponential sum in \eqref{eq:exp-sum} can be
employed, for instance matrix pencil methods
\cite{hua1990matrix,hua1991svd}, ESPRIT estimation methods
\cite{roy1986esprit}, and Cadzow denoising method
\cite{cadzow1988signal}.

\section{Exclusive phase retrieval}
\label{sec:excl-phase-retr}

In the following, we assume that $\Mat A \in \BC^{d \times d}$ is
diagonalizable, \ie\ $\Mat A = \Mat S \Mat \Lambda \Mat S^{-1}$.  For
a fixed signal $\Vek x \in \BC^d$ and a fixed sampling vector
$\Vek \phi \in \BC^d$, the given phaseless measurements are then of
the form
\begin{equation}
  \label{eq:phaseless-meas}
  \absn{\iProdn{\Vek x}{\Mat A^\ell \Vek \phi}}^2
  = \absn{\iProdn{\Vek y}{\Mat \Lambda^\ell \Vek \psi}}^2
  = \absbb{ \sum_{k=0}^{d-1}  \lambda_k^\ell
    \underbracket{\bar y_k \psi_k }_{\eqqcolon c_k}}^2
  = \sum_{j,k =0}^{d-1} c_j \bar c_k \, (\lambda_j \bar \lambda_k)^\ell,
\end{equation}
where $\Vek y \coloneqq \Mat S^* \Vek x$ and
$\Vek \psi \coloneqq \Mat S^{-1} \Vek \phi$.  Notice that
\eqref{eq:phaseless-meas} is an exponential sum with coefficients
$c_j \bar c_k$ and bases $\lambda_j \bar\lambda_k$.  In the following,
we require that the exponential sum has exactly $d^2$ unique bases.
Therefore, we call $M \coloneqq \{\mu_0, \dots, \mu_{d-1}\} \subset \BC$,
\begin{itemize}
\item \emph{collision-free} if the products $\mu_j \bar \mu_k$
  are pairwise distinct for $j,k \in \{0, \dots, d-1\}$.
\item \emph{absolutely collision-free} if $M$ is collision-free and if
  the products
  $\absn{\mu_j}\absn{\mu_k}$ are pairwise distinct for $j > k$.
\end{itemize}
Note that a matrix with collision-free eigenvalues is always
invertible, and that the eigenvalue decomposition becomes unique up to
permutations and global phases of the columns of $\Mat S$.
If the system or the matrix $\Mat A$ is known, we can usually recover
the signal $\Vek x$ using one sampling vector $\Vek \phi$.

\begin{theorem}
  \label{thm:rec-full}
  Let $\Mat A \in \BC^{d \times d}$ be known and diagonalizable with
  collision-free eigenvalues, and let $\Vek \phi \in \BC^d$ depend on
  all eigenspaces of $\Mat A$.  Then every $\Vek x \in \BC^d$ can be
  recovered from the samples
  $\{ \absn{\iProdn{\Vek x}{\Mat A^\ell \Vek \phi}} \}_{\ell
    =0}^{d^2-1}$ up to global phase.
\end{theorem}

\begin{proof}
  Assume that $\Mat A$ has the eigenvalue decomposition
  $\Mat A = \Mat S \Mat \Lambda \Mat S^{-1}$, and denote the
  coordinates of $\Vek \phi$ with respect to $\Mat S$ by
  $\Vek \psi \coloneqq \Mat S^{-1} \Vek \phi$.  The given measurements
  have the form
  \begin{equation*}
    \absn{\iProdn{\Vek x}{\Mat A^\ell \Vek \phi}}^2
    = \sum_{j,k =0}^{d-1} c_j \bar c_k \, (\lambda_j \bar \lambda_k)^\ell
  \end{equation*}
  with $c_k = \bar y_k \psi_k$ as shown in \eqref{eq:phaseless-meas}.
  Due to the distinctness of the products $\lambda_j \bar \lambda_k$,
  the coefficients $c_j \bar c_k$ may be calculated by solving a
  linear equation system based on an invertible Vandermonde matrix.
  The products $c_j \bar c_k$ contain the absolute values $\absn{c_k}$
  and the relative phases $\arg(c_j \bar c_k)$; so the
  factors $c_k$ are determined up to global phase.  Since the
  components of $\Vek \psi$ are non-zero, and since $\Mat S$ is
  invertible, we finally obtain $\Vek x$ up to global phase.  \qed
\end{proof}

\begin{corollary}
  \label{cor:rec-full}
  For almost all $\Vek a \in \BC^d$ and almost all $\Vek \phi \in
  \BC^d$, the signal $\Vek x \in \BC^d$ can be recovered from the
  samples $\{ \absn{\iProdn{\Vek x}{(\Circ \Vek a)^\ell \Vek \phi}}
  \}_{\ell = 0}^{d^2-1}$ up to global phase.
\end{corollary}

\begin{proof}
  The eigenvalues of $\Mat A\coloneqq \Circ \Vek a$ are just given by the discrete
  Fourier transform $\hat{\Vek a}$, and for almost all vectors $\Vek a \in
  \BC^d$ or, equivalently, $\hat{\Vek a} \in \BC^d$, the products
  $\hat a_j \bar{\hat a}_k$ are pairwise distinct.  Further, the
  vectors $\Vek \phi$ that are orthogonal to one column of the Fourier
  matrix form a hyperplane.  \qed
\end{proof}

We would like to note that phase retrieval from the sample
$\{ \absn{\iProdn{\Vek x}{(\Circ \Vek a)^\ell \Vek \phi_i}} \}_{\ell,i
  = 0}^{L-1,J-1}$ is possible with much less than $d^2$ temporal
measurement if more spatial measurement vectors $\Vek \phi_i$ and
polarization techniques are employed \cite{BH21}.

\section{Exclusive system identification}
\label{sec:excl-syst-ident}

The other way round, if the signal $\Vek x$ is known, then we can
usually identity the eigenvalues of the system
$\Mat A = \Mat S \Mat \Lambda \Mat S^{-1}$, \ie\ we assume that the
eigenvectors $\Mat S$ of the system are known.  For a convolutional
systems $\Mat A = \Circ \Vek a$, the eigenvectors are just the columns
of the Fourier matrix for instance.

\begin{theorem}
  \label{thm:sys-ident}
  Let $\Mat A = \Mat S \Mat \Lambda \Mat S^{-1}$ be diagonalizable by a known eigenvector basis $\Mat S$ and assume that the eigenvalues  are  collision-free.
  Let $\Vek \phi \in \BC^d$ depend on all eigenspaces of $\Mat A$, and
  let $\Vek x \in \BC^d$ be given.  If the coefficients $c_k$ defined
  in \eqref{eq:phaseless-meas} are collision-free too, then the
  eigenvalues $\lambda_0,\dots, \lambda_{d-1}$ of $\Mat A$ are defined
  by the samples
  $\{\absn{\iProdn{\Vek x}{\Mat A^\ell \Vek \phi}}\}_{\ell =
    0}^{2d^2-1}$ up to global phase.
\end{theorem}

\begin{proof}
  The measurements again have the form
  \begin{equation*}
    \absn{\iProdn{\Vek x}{\Mat A^\ell \Vek \phi}}^2
    = \sum_{j,k =0}^{d-1} c_j \bar c_k \, (\lambda_j \bar \lambda_k)^\ell
  \end{equation*}
  as shown in \eqref{eq:phaseless-meas}.  By assumption, the bases
  $\lambda_j \bar \lambda_k$ of this exponential sum are pairwise
  distinct and the coefficients $c_k$ are non-zero.  Thus the products $\lambda_j \bar \lambda_k$ and
  $c_j \bar c_k$ are determinable by Prony's method.  Note that
  Prony's method gives only the values but not the corresponding
  indices $j$ and $k$.  Exploiting that the products $c_j \bar c_k$
  are known -- $\Vek x$, $\Vek \phi$, $\Mat S$ are known, we can
  however deduce these indices.  Similarly to the proof of
  \thref{thm:rec-full}, the products $\lambda_j \bar \lambda_k$
  contain the absolute values $\absn{\lambda_k}$ and the relative
  phases $\arg(\lambda_j \bar \lambda_k)$; so the
  eigenvalues $\lambda_k$ are determined up to global phase.  \qed
\end{proof}

\begin{corollary}
  \label{cor:sys-ident}
  For almost all $\Vek x \in \BC^d$ and almost all
  $\Vek \phi \in \BC^d$, almost all kernels $\Vek a \in \BC^d$ can be
  recovered from the samples
  $\{ \absn{\iProdn{\Vek x}{(\Circ \Vek a)^\ell \Vek \phi}} \}_{\ell =
    0}^{2d^2-1}$ up to global phase.
\end{corollary}

\begin{proof}
  Again, the vectors $\Vek \phi$ that are orthogonal to one column of
  the Fourier matrix form a hyperplane.  Further, for almost all
  $\Vek \phi$ and $\Vek x$, the products $c_j \bar c_k$ in
  \eqref{eq:phaseless-meas} are pairwise distinct.  As discussed in
  the proof of \thref{cor:rec-full} almost all vectors
  $\Vek a \in \BC^d$ satisfy the assumption of \thref{thm:sys-ident}.
  \qed
\end{proof}

\section{Simultaneous phase \& system identification}
\label{sec:simult-phase-system}

If either the signal $\Vek x$ or the spectrum of $\Mat A$ are known,
we can recover the respective unknown information from the temporal
samples of only one sampling point.  To a certain degree, we may even
determine some information if  both -- the signal and the spectrum
-- are unknown.  Using one sampling point, we however lose the order
of the components. So we only obtain the unordered spectrum of $\Mat A$.

\begin{theorem}
  \label{thm:phase-eig-values}
  Let $\Mat A = \Mat S \Mat \Lambda \Mat S^{-1}$ be diagonalizable by a known eigenvector basis $\Mat S$ and assume that the eigenvalues are absolutely
  collision-free. Let $\Vek \phi \in \BC^d$ depend on
  all eigenspaces of $\Mat A$, and let
  $\Vek y \coloneqq \Mat S^* \Vek x$ be elementwise non-zero  for
  unknown $\Vek x \in \BC^d$.  Then the spectrum of $\Mat A$ is
  determined by the samples
  $\{\absn{\iProdn{\Vek x}{\Mat A^\ell \Vek \phi}}\}_{\ell =
    0}^{2d^2-1}$ up to global phase and winding direction.
\end{theorem}

\begin{proof}
  Since the coefficients $c_k = \bar y_k \psi_k$ with
  $\Vek y \coloneqq \Mat S^* \Vek x$ and
  $\Vek \psi \coloneqq \Mat S^{-1} \Vek \phi$ are non-zero, and since
  the eigenvalues are absolutely collision-free, the measurements have
  the form
  \begin{equation*}
    \absn{\iProdn{\Vek x}{\Mat A^\ell \Vek \phi}}^2
    = \sum_{j,k =0}^{d-1} c_j \bar c_k \, (\lambda_j \bar
    \lambda_k)^\ell
    = \sum_{k=0}^{d^2-1} \eta_k \beta_k^\ell
  \end{equation*}
  as shown in \eqref{eq:phaseless-meas}, where $\beta_k$ denotes the
  $d^2$ unique, unknown bases and $\eta_k$ the corresponding
  coefficients.  Applying Prony's method, we now recover the set
  $B \coloneqq \{\beta_k\}_{k=0}^{d^2-1}$.  Note that the relation
  between the elements of $B$ and $\{\lambda_j \bar \lambda_k: k,j=0,\dots,d-1\}$ is still
  unrevealed.

  In the following, we denote the recovered eigenvalues of $\Mat A$ in
  absolutely decreasing order by $\mu_k$, \ie\
  $\absn{\mu_0} > \cdots > \absn{\mu_{d-1}}$, and recover the permuted
  eigenvalues step by step.  Our assumption guarantees that
  $\mu_j \bar \mu_k$ differs from $\mu_k \bar \mu_j$, \ie\ the
  imaginary part cannot vanish; so the real values in $B$ correspond
  to the magnitudes $\absn{\mu_k}$.  The absolute collision freedom
  now allow us to recover the products $\mu_j \bar \mu_k$ and
  $\mu_k \bar \mu_j$ in $B$ corresponding to $\absn{\mu_j}$ and
  $\absn{\mu_k}$.  We now assume that $\mu_0$ is real and positive
  because the global phase cannot be recovered.  Considering
  $\mu_0 \bar \mu_1$ and $\mu_1 \bar \mu_0$, we obtain the relative
  phase $\arg(\mu_0) - \arg(\mu_1) \Mod 2 \uppi$ up to sign.  At this
  point, we have to chose one winding direction for the phase.  For
  $k = 2, \dots, d-1$, we may consider the relative phases between
  $\mu_k$ and the recovered $\mu_0$ and $\mu_1$, see
  \autoref{fig:proof-sys-id}, which uniquely determines the remaining
  phases.  \qed
\end{proof}

\begin{figure}
  \centering
  \includegraphics{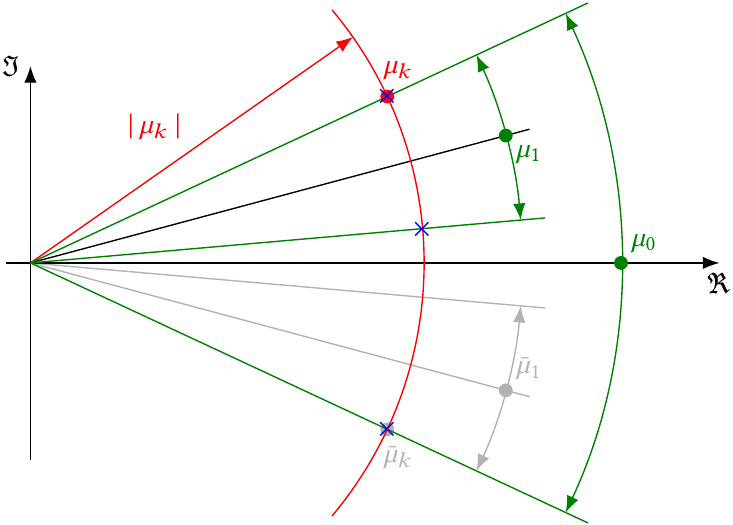}
  \caption{Propagating the phase in the proof of
    Theorem~\ref{thm:phase-eig-values}.  The points $\mu_0$ and
    $\mu_1$ are already known.  Using the relatives phases
    $\pm \arg(\mu_k \mu_0)$ and $\pm \arg(\mu_k \bar \mu_1)$, starting
    from $\mu_0$ and $\mu_1$, we obtain two possible candidates
    ({\color{Blue}$\times$}) for $\mu_k$ respectively since
    $\absn{\mu_k}$ is known too.  Further, since $\mu_1$ cannot also be
    real by assumption, exactly two candidates coincide yielding
    $\mu_k$.  For the other winding direction, \ie\ choosing
    $\bar \mu_1$ instead of $\mu_1$, we obtain $\bar \mu_k$.}
  \label{fig:proof-sys-id}
\end{figure}

\begin{remark}
  Note that the spectrum retrieved in \thref{thm:phase-eig-values} is an
  unordered set, \ie\ the relation to the known eigenvectors in
  $\Mat S$ is not revealed.  Applying the recovered relations between
  the bases, we may also recover the coefficients $c_k$ in
  \eqref{eq:phaseless-meas} up to global phase and winding direction.
  However, without knowing the actual order of the
  eigenvalues/coefficients, the recovery of the unknown signal is
  forlorn.  \qed
\end{remark}

Supposing that the unknown complex eigenvalues of the operator
$\Mat A$ have a clearly recognizable structure like
increasing/decreasing absolute values leads to highly artificial side
condition.  A nevertheless interesting special case are real-valued
convolutional systems with symmetrically decreasing kernels in the
frequency domain.  For the following theorem, we therefore restrict
the setup  to real-valued signals $\Vek x \in \BR^d$, real-valued
convolution operators $\Circ \Vek a$ with $\Vek a \in \BR^d$, and
real-valued sampling vectors $\Vek \phi \in \BR^d$.  We call a kernel
$\Vek a$ \emph{strictly, symmetrically decreasing} when
\begin{equation*}
  \hat{\Vek  a} \in \BR^d_{++}, \qquad
  \hat a_k = \hat a_{-k},
  \qquad\text{and}\qquad
  \hat a_k > \hat a_j
\end{equation*}
for $k,j \in \{0, \dots, \lfloor\nicefrac d2\rfloor \}$ with $k < j$.
The negative indices are here considered modulo $d$, and $\BR_{++}$
denotes the real and positive half axis.  Strictly, symmetrically
decreasing kernels correspond to low-pass filters, whose
identification in dynamical sampling has been studied in \cite{tang2017}.  Note that the signal $\Vek a$ is real and symmetric too.
We call the kernel \emph{collision-free in frequency} if the products
$\hat a_j \hat a_k$ are unique for
$k,j \in \{0, \dots, \lfloor\nicefrac d2\rfloor \}$ with $j \ge k$.
This definition differs from the collision-free complex sets.  In
order to recover both -- signal and kernel, we employ two sampling
vectors $\Vek \phi_1$ and $\Vek \phi_2$.  We call $\Vek \phi_1$ and
$\Vek \phi_2$ \emph{pointwise independent (in the frequency domain)}
when $\hat \phi_{1,k}$ and $\hat \phi_{2,k}$ interpreted as
two-dimensional real vectors are independent for
$k = 1, \dots, \lfloor \nicefrac d2 \rfloor$.  For this specific
setting, the identification of the system and the signal is usually
possible.

\begin{theorem}
  \label{thm:sys-sig:real-case}
  Let $\Vek a \in \BR^d$ be strictly, symmetrically decreasing and
  collision-free in frequency, let
  $\Vek \phi_1, \Vek \phi_2 \in \BR^d$ be pointwise independent, and
  let $\Vek x \in \BR^d$ satisfy
  $\Re[\bar{\hat x}_k\hat \phi_{i,k} ] \ne 0$ for $k =-\lfloor \nicefrac{(d-1)}2
  	\rfloor, \dots, \lfloor\nicefrac d2\rfloor$,
  $i = 1,2$.  Then $\Vek a$ and $\Vek x$ can be recovered from the
  samples
  \begin{equation*}
    \bigl\{
    \absn{\iProdn{\Vek x}{(\Circ \Vek a)^\ell \Vek \phi_1}},
    \absn{\iProdn{\Vek x}{(\Circ \Vek a)^\ell \Vek \phi_2}}
    \bigr\}_{\ell =0}^{L-1}
    \qquad\text{with}\qquad
    L \coloneqq
    \bigl( \bigl\lfloor\tfrac d2 \bigr\rfloor + 1 \bigr)
    \bigl( \bigl\lfloor\tfrac d2 \bigr\rfloor + 2 \bigr)
  \end{equation*}
  up to global sign.
\end{theorem}

\begin{proof}
  To simplify the notation, we first study the temporal samples with
  respect to an arbitrary sampling vector $\Vek \phi$.  Exploiting the
  symmetry of $\hat{\Vek a}$ and the conjugated symmetry of
  $\Vek c \coloneqq (\bar{\hat x}_k \hat \phi_k)_{k=-\lfloor \nicefrac{(d-1)}2
  	\rfloor}^{\lfloor\nicefrac d2\rfloor}$ caused by the Fourier transform, we
  combine the several times appearing bases in
  \eqref{eq:phaseless-meas} to obtain
  \begin{align*}
    \absn{\iProdn{\Vek x}{(\Circ \Vek a)^\ell \Vek \phi}}^2
    &= \absbb{
      \smashoperator[r]{\sum_{k=-\lfloor \nicefrac{(d-1)}2
      \rfloor}^{\lfloor\nicefrac d2\rfloor}} 
	c_k \hat a_k^\ell}^2  
      = \absbb{ \sum_{k=0}^{\lfloor \nicefrac d2\rfloor} 
      \gamma_k \Re[c_k ] \, \hat a_k^\ell}^2
    \\[\fskip]
    &= \sum_{k=0}^{\lfloor\nicefrac d2\rfloor}
      \sum_{j=0}^{\lfloor\nicefrac d2\rfloor}
      \gamma_k \gamma_j
      \Re[c_k ] \Re[c_j ] \, (\hat a_k \hat a_j)^\ell
    \\[\fskip]
    &= \sum_{k=0}^{\lfloor\nicefrac d2\rfloor}
      \sum_{j=k}^{\lfloor\nicefrac d2\rfloor}
      \underbrace{\gamma_{k,j}
      \Re[c_k ] \Re[c_j ] }_{\eta_k}\, (\hat a_k \hat a_j)^\ell
    = \sum_{k=0}^{\nicefrac L2-1}  \eta_k \beta_k^\ell,
  \end{align*}
with bases $\beta_k$ related to $\hat a_k \hat a_j$ and coefficients $\eta_k$ 
  where the multipliers are given by
  \begin{equation*}
    \gamma_k \coloneqq
    \begin{cases}
      1 &  \text{if} \; k=0 ,\\
      2 &  \text{if} \; k=1,\dots,\lfloor \nicefrac{(d-1)}2 \rfloor,\\
      1 &  \text{if} \; k=\nicefrac d2 \; \text{and $d$ is even},
    \end{cases}
    \qquad\text{and}\qquad
    \gamma_{k,j} \coloneqq
    \begin{cases}
      2 \gamma_k \gamma_j & \text{if} \; k \ne j,\\
      \gamma_k^2 & \text{if} \; k = j.
    \end{cases}
  \end{equation*}
  This exponential sum has exactly
  $\nicefrac12 (\lfloor\nicefrac d2 \rfloor + 1) (\lfloor\nicefrac d2
  \rfloor + 2)$ distinct bases since $\Vek a$ is collision-free in
  frequency.
   
  Applying Prony's method, we compute the bases $\beta_k$ and
  coefficients $\eta_k$.  Because the bases $\beta_k$ are all real and
  non-negative, we need a different procedure than before to reveal
  the relation to the factors $\hat a_k \hat a_j$.  Let $B$ be the set
  of recovered bases, where we assume $\beta_0 > \dots >
  \beta_{\nicefrac L2 - 1}$.  
  \begin{enumerate}[(i)]
  \item The strict, symmetric decrease of $\Vek a$ ensures
    $\beta_0 = \hat a_0^2$.  Now, remove $\beta_0$ from $B$.
  \item The next largest basis $\beta_1$ corresponds to $\hat
    a_0 \hat a_1$ allowing the recovery of $\hat a_1$.  Remove
    $\beta_1=\hat a_0 \hat a_1$ and $\hat a_1^2$ from $B$.
  \item The largest remaining bases correspond to $\hat a_0 \hat a_2$,
    which gives us $\hat a_2$.  Remove all products
    $\hat a_0 \hat a_2$, $\hat a_1 \hat a_2$, $\hat a_2 \hat a_2$ of
    $\hat a_2$ with the recovered components from $B$.
  \item Repeating this procedure, we obtain $\hat a_0, \dots, \hat
    a_{\lfloor \nicefrac d2 \rfloor}$ and, due to symmetry, the
    remaining half of $\hat{\Vek a}$.
  \end{enumerate}
  Alongside of the kernel, we also obtain the relation between
  $\eta_k$ and $\gamma_{k,j} \Re[c_{k}]\Re[c_{j}]$ for each sampling
  vector $\Vek \phi_1$, $\Vek \phi_2$.  Assuming
  $\Re[\bar{\hat x}_0 \hat\phi_{1,0}] = \bar{\hat x}_0 \hat\phi_{1,0}
  > 0$, we compute the real parts $\Re[\bar{\hat x}_k \hat\phi_{1,k}]$
  for $k=1,\dots,\lfloor \nicefrac d2 \rfloor$ by exploiting the
  revealed relative phases (sign changes), transfer the sign from
  $\bar{\hat x}_0 \hat\phi_{1,0}$ to
  $\bar{\hat x}_0 \hat\phi_{2,0} = \Re[\bar{\hat x}_0 \hat\phi_{2,0}]$
  since $\hat\phi_{1,0}$ and $\hat\phi_{2,0}$ are known, and determine
  $\Re[\bar{\hat x}_k \hat\phi_{2,k}]$ for
  $k=1,\dots,\lfloor \nicefrac d2 \rfloor$ analogously.  Due to the
  pointwise linear independence, the equation systems
  \begin{align*}
    \Re[\bar{\hat x}_k \hat \phi_{1,k}  ]
    &= \Re\hat \phi_{1,k} \Re\hat x_k + \Im\hat \phi_{1,k} \Im\hat x_k
      \\[\fskip]
    \Re[\bar{\hat x}_k \hat \phi_{2,k}]
    &= \Re\hat \phi_{2,k} \Re\hat x_k + \Im\hat \phi_{2,k} \Im\hat x_k
  \end{align*}
  gives us $\hat x_k$ for $k=1,\dots,\lfloor \nicefrac d2 \rfloor$.
  With the conjugated symmetry of $\hat{\Vek x}$, the inverse Fourier
  transform yields $\Vek x$ up to the sign.  \qed
\end{proof}

\begin{remark}
  Note that the assumption $\Re[\bar{\hat x}_k \hat \phi_{i,k}] \ne 0$
  for $k=0,\dots,d-1$ may be weakened to only hold for one sampling
  vector $\Vek \phi_1$ or $\Vek \phi_2$ as long as
  $\Re[\bar{\hat x}_0 \hat \phi_{i,0}] \ne 0$ for both.  In this case,
  the exponential sum corresponding to the temporal samples of the other
  sampling vector may consist of less than
  $\nicefrac12 \, (\lfloor \nicefrac d2 \rfloor +1)(\lfloor \nicefrac
  d2 \rfloor +2)$ bases.  Exploiting that the coefficients $\eta_k$ of
  the missing bases are zero, and spreading the sign between the
  non-zero coefficients, we can nevertheless recover $\Vek x$.  \qed
\end{remark}

It is also possible to identify the strictly, symmetrically decreasing
kernel alongside a complex signal and to allow complex sampling
vectors.  In this case, the temporal samples corresponding to one sampling
vector $\Vek \phi$ possesses the form
\begin{equation*}
  \absn{\iProdn{\Vek x}{(\Circ \Vek a)^\ell \Vek \phi}}^2
  =\sum_{k=0}^{\lfloor \nicefrac d2 \rfloor}
  \sum_{j=k}^{\lfloor \nicefrac d2 \rfloor}
  \tfrac{\gamma_{k,j}}4 \,
  \Re[(c_k + c_{-k})(\bar c_j + \bar c_{-j})]
  \, (\hat a_k \hat a_j)^\ell.
\end{equation*}
Similarly to the proof of \thref{thm:sys-sig:real-case}, we may
recover the kernel $\Vek a$ from the temporal samples of one sampling
vector if
\begin{equation}
  \label{eq:c-conv-sys:coeff}
  \Re \bigl[ (\bar{\hat x}_k \hat \phi_{k}
  + \bar{\hat x}_{-k} \hat \phi_{-k})(
  \overline{\bar{\hat x}_j \hat \phi_{j}
  + \bar{\hat x}_{-j} \hat \phi_{-j}} )
  \bigr] \ne 0
\end{equation}
for $k,j = 0, \dots, \lfloor \nicefrac d2 \rfloor$.  Additionally, the
signal $\Vek x$ may be recovered if four sampling vectors are
employed.  In this case, the coefficient of $\hat a_0^2$ is just
$\absn{\bar{\hat x}_0 \hat \phi_{i,0}}$; so fixing the phase for
$c_{1,0}$, we may spread the phase to $c_{i,0}$, $i=2,3,4$, where the
first index stands for the related sampling vector, \ie\ all
coefficients $c_{i,0}$ are known.  If the equation system
\begin{align*}
  \Re [ c_{i,0} (\bar c_{i,k} + \bar c_{i,-k}) ]
  &= \Re[\bar c_{i,0} \hat \phi_{i,k}] \, \Re \hat x_k
    +  \Im[\bar c_{i,0} \hat \phi_{i,k}] \, \Im \hat x_k
  \\[\fsmallskip]
  & \qquad+ \Re[\bar c_{i,0} \hat \phi_{i,-k}] \, \Re \hat x_{-k}
    +  \Im[\bar c_{i,0} \hat \phi_{i,-k}] \, \Im \hat x_{-k}
\end{align*}
with $i=1,\dots,4$ is solvable, we obtain $\hat{\Vek x}$ and thus
$\Vek x$.  Notice that the recovery of $\hat{\Vek a}$ here is not a 
special case of \thref{thm:phase-eig-values} since $\hat{\Vek a}$ is
not collision-free as a complex set.  In sum, the following statement
can be  established.

\begin{theorem}
  \label{thm:sys-sig:comp-case}
  Let $\Vek a \in \BR^d$ be strictly, symmetrically decreasing and
  collision-free in frequency, let
  $\Vek \phi_1, \dots, \Vek \phi_4 \in \BC^d$ and $\Vek x \in \BC^d$
  satisfy \eqref{eq:c-conv-sys:coeff}.  If the real-valued vectors
  \begin{equation*}
    (\Re[\bar c_{i,0} \hat \phi_{i,k}], \Im[\bar c_{i,0} \hat
    \phi_{i,k}], \Re[\bar c_{i,0} \hat \phi_{i,-k}], \Im[\bar c_{i,0}
    \hat \phi_{i,-k}])^\T, \qquad i=1,\dots,4,
  \end{equation*}
  are independent for each
  $k = 1, \dots, \lfloor \nicefrac {(d-1)}2 \rfloor$, then $\Vek a$ and
  $\Vek x$ can be recovered from the samples
  \begin{equation*}
    \bigl\{
    \absn{\iProdn{\Vek x}{(\Circ \Vek a)^\ell \Vek \phi_i}}
    \bigr\}_{\ell =0, i=1}^{L-1,4}
    \qquad\text{with}\qquad
    L \coloneqq
    \bigl( \bigl\lfloor\tfrac d2 \bigr\rfloor + 1 \bigr)
    \bigl( \bigl\lfloor\tfrac d2 \bigr\rfloor + 2 \bigr)
  \end{equation*}
  up to global phase.
\end{theorem}

\begin{remark}
  The strictly, symmetrically decreasing kernels form a
  $(\lfloor \nicefrac d2 \rfloor +1)$-dimensional manifold.  Further,
  the not collision-free kernels live on the union of submanifolds
  with strictly smaller dimension; so almost all strictly,
  symmetrically decreasing kernels are collision-free.  Moreover,
  almost all vectors $\Vek x$ and $\Vek \phi_i$ satisfy the posed
  conditions in the real as well as in the complex setting.  \qed
\end{remark}

\section{Multiple sampling vectors}
\label{sec:multi-samp-vec}

Let us return to the parameter identification of arbitrary systems
after that brief digression to strictly, symmetrically decreasing
convolution kernels.  Revisiting the statement in
\thref{thm:phase-eig-values}, our main problem has been that we cannot
recover the order of the spectrum from merely one sampling vector if
both -- signal and eigenvalues -- are unknown. Since our analysis is
based on Prony's method, we have always relied on a squared number of
measurements.  To surmount these shortcomings, we suppose specifically
constructed sets of sampling vectors.

Instead of assuming that the sampling vectors $\Vek \phi_i$ depend on
all eigenspaces of the system matrix $\Mat A$, we now assume that
$\Vek \phi_i$ might only depends on a small set of eigenspaces.  Considering
the temporal samples for such a sampling vector, in analogy to
\eqref{eq:phaseless-meas}, we have
\begin{equation*}
  \absn{\iProdn{\Vek x}{\Mat A^\ell \Vek \phi_i}}^2
  = \absn{\iProdn{\Vek y}{\Mat \Lambda^\ell \Vek \psi_i}}^2
  = \absbb{ \sum_{k \in \mathcal I_i} \lambda_k^\ell
    \underbracket{ \bar y_k \psi_{i,k}}_{\eqqcolon c_{i,k}}}^2
  = \sum_{j,k \in \mathcal I_i} c_{i,j} \bar{c}_{i,k} \,
  (\lambda_j \bar{\lambda}_k)^\ell,
\end{equation*}
where $\Vek y \coloneqq \Mat S^* \Vek x$,
$\Vek \psi_i \coloneqq \Mat S^{-1} \Vek \phi_i$, and
$\mathcal I_i \coloneqq \supp \Vek \psi_i$.  Since $\Vek \phi_i$ only
captures a small part of the spectrum, the last sum only consists of
$\absn{\mathcal I_i}$ exponentials instead of $d^2$ and allows the
recovery of a specific part of the spectrum.  To combine these partial
information and to overcome the mentioned issues, the sampling vectors
$\{ \Vek \phi_i \}_{i=0}^{J-1}$ with
$\Vek \psi_i \coloneqq \Mat S^{-1} \Vek \phi_i$ should allow
\begin{enumerate}[(i)]
\item \emph{index separation:} the supports of
  $\{\Vek\psi_i\}_{i=0}^{J-1}$ form a full cover meaning
  $\bigcup_{i=0}^{J-1} \supp \Vek \psi_j = \{0, \dots, d-1\}$, and  for every $k \in \{0, \dots, d-1\}$
  there exist two index sets ${\mathcal F}_k$ and
  ${\mathcal G}_k$ such that
  \begin{equation}
    \label{eq:ind-sep}
    \{k\} = \bigcap_{i\in{\mathcal F}_k} \supp \Vek \psi_i
    \Bigm\backslash \bigcup_{i \in  {\mathcal G}_k}\supp \Vek \psi_i,
  \end{equation}
 
\item \emph{phase propagation:}  the set $\{ \Vek
  \phi_i\}_{i=0}^{J-1}$ is ordered such that
  \begin{equation}
    \label{eq:phase-prop}
    \#\biggl[\supp \Vek \psi_k \cap \bigcup_{i=0}^{k-1} \supp \Vek
    \psi_i\biggr] = 2 
  \end{equation}
  for $k = 1, \dots, J-1$, \ie\ there is an overlap of two elements at
  least,
\item \emph{winding direction determination:} there are indices $i_1$,
  $i_2$, $k_1$, $k_2$ such that
  \begin{equation}
    \label{eq:wind-dir-det}
    \arg(\psi_{i_1,k_1} \bar \psi_{i_1,k_2})
    \not\equiv
    \arg(\psi_{i_2,k_1} \bar \psi_{i_2,k_2})
    \Mod \uppi,
  \end{equation}
  where $\psi_{i_1,k_1} \bar \psi_{i_1,k_2}$ and
  $\psi_{i_2,k_1} \bar \psi_{i_2,k_2}$ are non-zero.
\end{enumerate}
If the sampling vectors $\Vek \phi_i$ fulfill all three assumptions, we say that the
sampling set allows \emph{parameter identification and phase
  retrieval} (up to global phase).

\begin{theorem}
  \label{thm:sig-sys:sampl-set}
  Let $\Mat A = \Mat S \Mat \Lambda \Mat S^{-1}$ be diagonalizable by a known eigenvector basis $\Mat S$ and assume that the eigenvalues are absolutely
  collision-free. Let $\{\Vek\phi_j\}_{j=0}^{ J-1}\subset \BC^d$
  allow parameter identification and phase retrieval, and let
  $\Vek y \coloneqq \Mat S^* \Vek x$ be elementwise non-zero for
  unknown $\Vek x \in \BC^d$. Then the eigenvalues
  $\lambda_0, \dots, \lambda_{d-1}$ of $\Mat A$ and the signal
  $\Vek x$ are determined by the spatiotemporal samples
  \begin{equation*}
    \bigl \{ | \langle x, \Mat A^\ell \Vek \phi_i \rangle
    | \bigr\}_{\ell,i=0}^{L_i^2-1,J-1}
    \qquad\text{with}\qquad
    L_i \coloneqq  \#[\supp (\Mat S^{-1} \Vek \phi_i)]
  \end{equation*}
  up to  global phase. 
  \end{theorem}
\begin{proof}
  Using the procedure in the proof of \thref{thm:phase-eig-values}, we
  recover the unblocked part
  $\Lambda_i \coloneqq \{\lambda_k : k \in \mathcal I_i\}$ of the spectrum of
  $\Mat A$ for each $i =0, \dots, J-1$ up to global phase and winding
  direction.  Note that we do not know which value in $\Lambda_i$
  corresponds to which index.  However, since the eigenvalues are
  absolutely collision-free, and since the sampling set allows index
  separation, we have
  \begin{equation*}
   \bigcap_{j \in \mathcal F_k} \absn{\Lambda_i} \Bigm \backslash \bigcup_{i
     \in \mathcal G_k} \absn{\Lambda_i}
   = \absn{\lambda_k},
  \end{equation*}
  where the absolute value is applied element by element.  Thus the
  true index of the eigenvalues is revealed.

  Using that the sampling set allows phase propagation, we align the
  global phase and winding direction of the sets $\Lambda_i$ as
  follows.  First, we fix the global phase and winding direction of
  $\Lambda_0$.  There are at least two eigenvalues $\lambda_{k_1}$ and
  $\lambda_{k_2}$ that are contained in $\Lambda_0$ and $\Lambda_1$.
  The collision-freedom ensures
  $\arg(\lambda_{k_1} \bar \lambda_{k_2}) \not\equiv 0 \Mod \uppi$.
  Using $\lambda_{k_1}$ and $\lambda_{k_2}$, which can be identified
  by their absolute values, the global phase and winding direction are
  uniquely transferable form $\Lambda_0$ to $\Lambda_1$, \ie\ we
  obtain the eigenvalues in $\Lambda_0 \cup \Lambda_1$ up to global
  phase and winding direction.  Repeating this argument, we
  propagate the phase information to the remaining subsets
  $\Lambda_i$, which results in the recovery of all eigenvalues
  $\lambda_0, \dots, \lambda_{d-1}$ up to global phase and winding
  direction.

  The ambiguity with respect to the winding direction occurs since we
  have not been able to determine whether the true relative phase
  between $\lambda_{j}$ and $\lambda_{k}$ corresponds to
  $\arg(\lambda_{j} \bar \lambda_{k})$ or to
  $\arg(\lambda_{k} \bar \lambda_{j})$.  Let us now consider the
  indices $i_1$, $i_2$, $k_1$, $k_2$ in the winding direction property
  \eqref{eq:wind-dir-det} of $\{\Vek \phi_i \}_{i=0}^{J-1}$.  Notice
  that both $\lambda_{k_1}$ and $\lambda_{k_2}$ are captured by the
  sampling vectors $\Vek \phi_{i_1}$, $\Vek \phi_{i_2}$.  Due to the
  missing winding direction, the coefficients
  $c_{i_1,k_1}\bar{c}_{i_1,k_2}$ and $c_{i_2,k_1}\bar{c}_{i_2,k_2}$
  can only be identified up to the conjugation; so we merely obtain
  $\Re[c_{i_1,k_1}\bar{c}_{i_1,k_2}]$ and
  $\Re[c_{i_2,k_1}\bar{c}_{i_2,k_2}]$, which however are given by
  \begin{align*}
    \Re[ c_{i_1,k_1}\bar{c}_{i_1,k_2}]
    &= \Re[y_{k_1} \bar  y_{k_2} ] \, \Re[\psi_{i_1,k_1} \bar
      \psi_{i_1,k_2}]
      + \Im[y_{k_1} \bar y_{k_2}] \,
      \Im[\psi_{i_1,k_1} \bar \psi_{i_1,k_2}],
    \\[\fskip]
    \Re[ c_{i_2,k_1}\bar{c}_{i_2,k_2}]
    &= \Re[y_{k_1} \bar  y_{k_2} ] \, \Re[\psi_{i_2,k_1} \bar
      \psi_{i_2,k_2}]
      + \Im[y_{k_1} \bar y_{k_2}] \,
      \Im[\psi_{i_2,k_1} \bar \psi_{i_2,k_2}].
  \end{align*}
  Our assumptions guarantees that this equation system has the unique
  answer $y_{k_1} \bar y_{k_2}$, which yields $c_{i_1,k_1}\bar{c}_{i_1,k_2}$
  and $c_{i_2,k_1}\bar{c}_{i_2,k_2}$ without conjugation ambiguity.
  Further, at least one of the products $c_{i_1,k_1}\bar{c}_{i_1,k_2}$
  and $c_{i_2,k_1}\bar{c}_{i_2,k_2}$ has a non-vanishing imaginary
  part again due to \eqref{eq:wind-dir-det}.  The corresponding basis
  $\lambda_{k_1} \bar \lambda_{k_2}$ reveals the true winding
  direction resulting in the recovery of $\lambda_0, \dots,
  \lambda_{d-1}$ up to global phase.

  Considering the coefficient of the temporal samples for each
  $\Vek \phi_i$, we determine $y_k$ with $k \in \supp \Vek \psi_i$ up
  to global phase.  The recovered components of $\Vek y$ may now be
  aligned due to the overlap between the supports in
  \eqref{eq:phase-prop} yielding $\Vek y$ up to global phase.
  Applying the inverse of $\Mat S^*$, we finally obtain the wanted
  signal $\Vek x$ up to global phase.  \qed
\end{proof}

\begin{remark}
  \label{rem:sig-sys:sampl-set}
  The absolute collision-freedom of the eigenvalues can be weakened.
  More precisely, we only require the absolute collision-freedom on
  the non-blocked parts of the spectrum with respect to $\{\Vek
  \phi_i\}_{i=0}^{J-1}$, \ie\ we only require that the sets
  $\Lambda_i$ are absolutely collision-free.  In order to propagate
  the phase, there have to be to at least two indices
  \begin{equation*}
    k_1, k_2 \in \supp \Vek \psi_k \cap
    \bigcup_{i=0}^{k-1} \supp \Vek \psi_i 
  \end{equation*}
  for $k = 1, \dots, J-1$, \cf\ \eqref{eq:phase-prop}, satisfying
  $\arg(\lambda_{k_1} \bar \lambda_{k_2}) \not\equiv 0 \Mod \uppi$.
  \qed
\end{remark}

\thref{thm:sig-sys:sampl-set} not only allow us to recover the signal
and the system's eigenvalues simultaneously but also to reduce the
required number of samples.  In the statements before, the number of
measurements to apply Prony's method is always a multiple of the
squared dimension, \ie\ we require $\Landau(d^2)$ samples.  In
\thref{thm:sig-sys:sampl-set} the number of spatiotemporal samples
mainly correlate with the support sparsity
$L_i \coloneqq \# [\supp(\Mat S^{-1} \Vek \phi_i)]$.  With
$L \coloneqq \max\{L_i : i = 0, \dots J-1\}$, the number of samples is
thus bounded by $2 L^2 J$.  Notice that we need $d$ vectors at the
most to build a sampling set allow parameter identification and phase
retrieval. For instance the sampling vectors may be constructed such
that $\supp \Vek \psi_i \coloneqq \{i, \dots, i + L - 1\}$ for
$i = 0, \dots, d-L$ and $L \ge 3$. We then employ only $\Landau(d L^2)$
measurement.  For a fixed sparsity $L$, we only need
linearly many spatiotemporal samples.

\begin{corollary}
  Under the assumption of \thref{thm:sig-sys:sampl-set}, the
  eigenvalues of $\Mat A \in \BC^{d \times d}$ and the unknown signal
  $\Vek x \in \BC^d$ are identifiable with $\Landau(d)$
  spatiotemporal samples.
\end{corollary}

The idea of blocking a part of the spectrum to reduce the number of
required spatiotemporal samples clearly transfers to
\thref{thm:sys-sig:real-case} and \ref{thm:sys-sig:comp-case}.  The
indices of the recovered eigenvalues is then determined by the strict,
symmetrical decay; so the index separation, phase propagation, and
winding direction determination is not required, although the supports
of $\{\hat{\Vek \phi}_i\}_{i=0}^{J-1}$ should still form a full cover.
Considering \thref{thm:sys-sig:real-case} exemplarily, we instead need
that, for every $k \in \{0, \dots, d-1\}$, there exists at least one
index $j \in \{0, \dots, J-1\}$ such that
$\Re [\bar{\hat x}_k \hat \phi_{i,k}] \ne 0$ to recover all components
of $\hat{\Vek a}$ and two indices $i_1, i_2 \in \{0, \dots, J-1\}$
such that $\hat\phi_{i_1,k}$ and $\hat\phi_{i_2,k}$ are linearly
independent interpreted as two-dimensional real vectors to recover all
components of $\hat{\Vek x}$.

\section{Phase \& system identification in infinite dimensions}
\label{sec:infinite-dim}

Up to this point, we only considered the finite-dimensional setting.
The central ideas to apply Prony's method to identify the eigenvalues
of the system and the unknown signal simultaneously is however
extendable to the infinite-dimensional setting too.  In the following,
we consider an infinite-dimensional, complex Hilbert space $\Hil H$
and call an invertible, bounded, linear operator
$\Op A : \Hil H \to \Hil H$ \emph{diagonalizable} if $\Op A$ can be
factorized into $\Op A = \Op S \Lambda \Op S^{-1}$, where
$\Op S \colon \ell^2(\mathcal Z) \to \Hil H$ is an invertible,
bounded, linear operator,
$\Lambda \colon \ell^2(\mathcal Z) \to \ell^2(\mathcal Z)$ is a
multiplication operator, and $\mathcal Z$ is an infinite countable set like
$\BN$ or $\BZ$.  The elementwise \emph{multiplication operator}
$\Lambda \colon \ell^2(\mathcal Z) \to \ell^2(\mathcal Z)$ is defined
by
\begin{equation*}
  \Lambda( y) \coloneqq \bigl( \lambda_k \, y_k \bigr)_{k \in
    \mathcal Z}
\end{equation*}
with bounded eigenvalues $\lambda_k \in \BC_*$, \ie\
$\sup_{k \in \mathcal Z} \absn{\lambda_k} < \infty$.

Similarly to the finite-dimensional setting, the temporal samples for
one sampling vector $\Vek \phi_i$ are given by
\begin{equation*}
  \absn{\iProdn{ x}{\Op A^\ell \phi_i}_{\Hil H}}^2
  = \absn{\iProdn{ y}{\Lambda^\ell  \psi_i}_{\ell^2(\mathcal Z)}}^2
  = \absbb{ \sum_{k \in \mathcal I_i} \lambda_k^\ell
    \underbracket{ \bar y_k \psi_{i,k}}_{\eqqcolon c_{i,k}}}^2
  = \sum_{j,k \in \mathcal I_i} c_{i,j} \bar{c}_{i,k} \,
  (\lambda_j \bar{\lambda}_k)^\ell,
\end{equation*}
where $ y \coloneqq \Op S^*  x$,
$\psi_i \coloneqq \Op S^{-1}  \phi_i$, and
$\mathcal I_i \coloneqq \supp \psi_i \subset \mathcal Z$.  
If $\supp  \psi_i$ is finite, the sum on the right-hand side
becomes finite such that Prony's method may be applied to recover the
present eigenvalues (without indices).  In order to determine the
complete spectrum, the finite supports of $ \phi_i$ have to form a full
cover of $\mathcal Z$, which is only possible for infinitely many
sampling vectors, \ie\ $J = \infty$.  To align the recovered subsets,
we rely again on the parameter identification and phase retrieval
properties in (\ref{eq:ind-sep}--\ref{eq:wind-dir-det}).  In sum, we
obtain the following recovery guarantee for infinite-dimensional
Hilbert spaces.

\begin{theorem}
  \label{thm:sig-sys:hilbert}
  Let $\Op A \colon \Hil H \to \Hil H$ with absolutely collision-free
  eigenvalues be diagonalizable by a known
  $\Op S \colon \ell^2(\mathcal Z) \to \Hil H$, where $\Hil H$ is an
  infinite-dimensional Hilbert space and $\mathcal Z$ an infinite countable set.
   Let $\{\phi_j\}_{j=0}^{\infty} \subset \Hil H$ allows parameter
  identification and phase retrieval with finitely supported
  $\Op S^{-1}   \phi_i$, and let $ y \coloneqq \Op S^*  x$
  be elementwise non-zero for unknown $ x \in \Hil H$. Then the
  eigenvalues $\lambda_k$ with $k \in \mathcal Z$ of $\Op A$ and the
  signal $ x$ are defined by the spatiotemporal samples
  \begin{equation*}
    \bigl \{ | \langle x, \Op A^\ell  \phi_i \rangle
    | \bigr\}_{\ell,i=0}^{L_i^2-1,\infty}
    \qquad\text{with}\qquad
    L_i \coloneqq  \#[\supp (\Op S^{-1}  \phi_i)]
    \addmathskip
  \end{equation*}
  up to a global phase.
\end{theorem}

Since the statement can be established with the construction in the
proof of \thref{thm:sig-sys:sampl-set}, we omit the proof.
Furthermore, \thref{rem:sig-sys:sampl-set} carries over to the
infinite-di\-men\-sion\-al setting as well.  Note that the non-zero
assumption on $ y \coloneqq \Op S^* x$ is crucial since
otherwise a part of the spectrum is blocked in all spatiotemporal
measurements and thus cannot be recovered.

An example for the infinite-dimensional Hilbert space setting is the
repeated convolution of periodic function.  For this, let $\Hil H$ be
the Hilbert space $L^2(\BT)$ of all square-integrable, one-periodic
functions on the torus $\BT$.  The convolution operator with respect
to an absolutely integrable function $a \in L^1(\BT)$ is defined by
\begin{equation*}
  \conv_a [\phi](t) \coloneqq (a * \phi)(t)
  = \int_{\BT} a(t-s) \, \phi(s) \diff s
\end{equation*}
for $t \in \BT$.  The convolution operator $\conv_a$ is here an
isomorphism on $L^2(\BT)$ due to Young's convolution
inequality, see e.g., \cite{plonka2018numerical}, and is diagonalized by the finite
Fourier transform $\Fourier \colon L^2(\BT) \to \ell^2(\BZ)$ given by
\begin{equation*}
  \Fourier [\phi](k) \coloneqq
  \hat \phi(k) \coloneqq
  \int_{\BT} \phi(t) \, \e^{-2\uppi \I k
    t} \diff t.
\end{equation*}
More precisely, we have $\Op S^{-1} = \Fourier$, $\mathcal Z =
\BZ$, and $\Lambda \colon \psi \mapsto \hat a \odot \psi$, where
$\odot$ denotes the elementwise multiplication.  Due to the support
constraints on the Fourier coefficients, the sampling vectors
$\{\phi_i\}_{i=0}^\infty$ are trigonometric polynomials.

\begin{corollary}
  \label{thm:sig-sys:torus}
  Let $a \in L^1(\BT)$ with absolutely collision-free Fourier
  coefficients $\hat a$ be unknown, let
  \raisebox{0pt}[0pt][0pt]{$\{\phi_j\}_{j=0}^{\infty}$} be a set
  of trigonometric polynomials allowing parameter identification and
  phase retrieval, and let \raisebox{0pt}[0pt][0pt]{$\hat f$} be
  elementwise non-zero for unknown $f \in L^2(\BT)$. Then $a$ and $f$
  are defined by the spatiotemporal samples
  \begin{equation*}
    \bigl \{ | \langle f, \conv_a^\ell[\phi_i] \rangle
    | \bigr\}_{\ell,i=0}^{L_i^2-1,\infty}
    \qquad\text{with}\qquad
    L_i \coloneqq  \#[\supp (\hat \phi_i)]
  \end{equation*}
  up to  global phase.
\end{corollary}

The proposed eigenvalue and signal identification can be generalized
to arbitrary Banach spaces $\Hil X$ that are isomorphic to a sequence
space like $\ell^p(\mathcal Z)$.  In this case, the inner products
have to be replaced by appropriate dual pairings.

\section{Sensitivity analysis}
\label{sec:sens-analysis}

In the previous sections, we have shown that the dynamical phase
retrieval and system identification problem is solvable under certain
assumptions from exact measurements.  In the following, we study the
situation for disturbed measurements.  Since our constructive proofs
have been heavily based on Prony's method, the sensitivity also mainly
depends on it.  On the bases of Potts \& Tasche \cite{PT10},
initially, the sensitivity of the approximate Prony method is
considered; hereby, we follow the proofs of \cite{PT10} for
real-valued exponential sums and generalize to the complex setting.
In a second step, we analyse the error propagation in dynamical phase
retrieval.

\subsection{Sensitivity of Prony's method}
\label{sec:sens-prony}

Essentially, the (approximate) Prony method is a two step approach to
determine the parameters of the exponential sum \eqref{eq:exp-sum}.
In the first step, the unknown bases $\Vek \beta$ are recovered using
a singular value decomposition and determining the roots of the Prony
polynomial.  In the second, the unknown coefficients $\Vek \eta$ are
computed by solving a linear least-square problem.  To analyse the
sensitivity of the first step, we require the following lemma
estimating the norm of a rectangular Vandermonde matrix by the maximal
radius of the bases 
\begin{equation*}
  \rho_{\Vek \beta} \coloneqq \max \{1, \pNormn{\Vek \beta}_\infty\}.
\end{equation*}

\begin{lemma}
  \label{lem:norm-vander}
  For $\Vek \beta \in \BC^K$, the Vandermonde matrix
  $\Mat V_L(\Vek \beta)$ satisfies
  \begin{equation*}
    \pNormn{\Mat V_L(\Vek \beta)}_{\infty} \leq K \rho_{\Vek
    \beta}^{L-1},
    \qquad
    \pNormn{\Mat V_L(\Vek \beta)}_{1} \leq L \rho_{\Vek
      \beta}^{L-1},
  \end{equation*}
  and thus
  \begin{equation*}
    \pNormn{\Mat V_L(\Vek \beta)}_2
    \le \sqrt{KL} \rho_{\Vek \beta}^{L-1}. 
  \end{equation*}
\end{lemma}

\begin{proof}
  The assertion immediately follows from
  \begin{align*}
    &\| \Mat V_L (\Vek \beta) \|_\infty
    \leq \max_{0 \le \ell < L} \sum_{k=0}^{K-1} |\beta_k|^\ell
      \leq K \max_{0 \le \ell < L} \| \Vek \beta \|_\infty^{\ell}
      \leq K \max\{ 1, \pNormn{\Vek \beta}^{L-1}_\infty \},
    \\[\fskip]
    &\| \Mat V_L (\Vek \beta) \|_1
      \leq \max_{0 \le k < K} \sum_{\ell=0}^{L-1} |\beta_k|^\ell
      \leq L \max_{0 \le k < K} \Bigl(
      \max\{1, \beta_k^{L-1} \} \Bigr)
      \leq L \max\{ 1, \pNormn{\Vek \beta}^{L-1}_\infty \},
    \\[\fskip]
    &\pNormn{\Mat V_L(\Vek \beta)}_2
      \le \sqrt{\pNormn{\Mat V_L(\Vek
      \beta)}_1\pNormn{\Mat V_L(\Vek \beta)}_\infty}.
      \tag*{\qed}
  \end{align*}
\end{proof}

Further, we need a left inverse of the rectangular Vandermonde matrix.
The inverse of a quadratic Vandermonde matrix has been well studied in
the literature \cite{MS58,Tur66,Gau62,Gau75,EP81,ElM03,Pan16,HCP19}
and is given by
\begin{equation}
  \label{eq:inv-vander}
  \Vek V^{-1}(\Vek \beta) =
  \Bigl( (-1)^{K-k-1} \, S^{(\ell)}_{K-k-1} (\Vek \beta)
  \Bigm\slash \Pi_\ell (\Vek \beta) \Bigr)_{\ell,k=0}^{K-1},
\end{equation}
where $S_k^{(\ell)}$ denotes the $k$th elementary symmetric polynomial
without the $\ell$th variable, which is more precisely defined by
\begin{equation*}
  S_k^{(\ell)}(\Vek \beta)
  = \smashoperator{
    \sum_{\substack{
        \hspace{10pt}
        0 \le j_1 < \cdots < j_k \le K-1
        \\
        \hspace{11pt}
        j_1,\dots,j_k \ne \ell
      }}} 
  \beta_{j_1} \dots \beta_{j_k}
  \qquad\text{and}\qquad
  S_0^{(\ell)}(\Vek \beta) = 1,
\end{equation*}
and where $\Pi_\ell$ is the product of differences
\begin{equation*}
  \Pi_\ell (\Vek \beta) \coloneqq \prod_{\substack{k=0\\k \ne \ell}}^{K-1}
  (\beta_\ell - \beta_k).
\end{equation*}
The classical elementary symmetric polynomials are based on all
elements of $\Vek \beta$, \ie\ without the condition $j_1,\dots, j_k
\ne \ell$, and are denoted by $S_k(\Vek \beta)$.

\begin{lemma}[Gautschi \cite{Gau62}]
  \label{lem:ele-sym-poly}
  The elementary symmetric polynomial are bounded by
  \begin{equation*}
    \sum_{k=0}^{K-1} \absn{S_k(\Vek \beta)}
    \le \prod_{k=0}^{K-1} ( 1 + \absn{\beta_k}).
  \end{equation*}
\end{lemma}

\begin{proof}
  For convenience, we give the brief proof  from \cite{Gau62}.  On the
  bases of Vieta's formula, the elementary symmetric polynomials
  are related to the polynomial
  \begin{equation*}
    z \mapsto
    \sum_{k=0}^{K-1} (-1)^k \, S_k(\Vek \beta) \, z^{K-k-1}
    = \prod_{k=0}^{K-1} (z - \beta_k).
  \end{equation*}
  Choosing $z = -1$, we obtain the assertion for real and positive
  $\beta_k$, $k = 0, \dots, K-1$. The general assertion then follows
  from $\absn{S_k(\Vek \beta)} \le S_k(\absn{\Vek \beta})$, where
  $\absn{\cdot}$ is applied elementwise.  \qed
\end{proof}

Defining the product radius $\pi_{\Vek \beta}$ and the minimal separation
$\sigma_{\Vek \beta}$ of the bases in $\Vek \beta$ as
\begin{equation*}
  \pi_{\Vek \beta} \coloneqq \prod_{k=0}^{K-1} ( 1 + \absn{\beta_k} )
  \qquad \text{and} \qquad
  \sigma_{\Vek \beta} \coloneqq
  \min \{ \absn{\beta_\ell - \beta_k} : 0 \le \ell < k \le K-1 \},
\end{equation*}
the norm of the inverse Vandermonde matrix is bounded as follows.

\begin{proposition}
  \label{prop:norm-inv-vander}
  For $\Vek \beta \in \BC_*^K$ with distinct elements, the inverse of
  the quadratic Vandermonde matrix $\Mat V(\Vek \beta)$ satisfies
  \begin{equation*}
    \pNormn{\Mat V^{-1}(\Vek \beta)}_\infty
    \le \frac{\pi_{\Vek \beta}}{\sigma_{\Vek \beta}^{K-1}}.
  \end{equation*}
\end{proposition}

\begin{proof}
  The bound follows immediately from the inversion formula
  \eqref{eq:inv-vander} and from applying \thref{lem:ele-sym-poly} to
  the sum over the elementary symmetric polynomials $S_k^{(\ell)}$
  with fixed $\ell$ as well as multiplying the estimated for the row
  sums by the missing factor $(1 + \absn{\beta_\ell}) > 1$.  \qed
\end{proof}

The norm estimates regarding the Vandermonde matrix allow us to study
the quality of the Prony polynomial for perturbed measurements.  If
the error is small, the true bases are nearly roots; so we may hope
that the first two steps of \thref{alg:app-prony} approximate the
bases well.  Recall that the approximate Prony method is based on the
assumption that the measurement error $\epsilon$ with
$\absn{h_\ell + e_\ell} \le \epsilon$ is small enough such that the
singular values of the unperturbed Hankel matrix fulfil
$\sigma_k(\Mat H) \ge 2 \pNormn{\Mat E}_2$.  The spectral norm is here
bounded by
\begin{equation*}
  \pNormn{\Mat E}_2
  \le \sqrt{\pNormn{\Mat E}_1 \pNormn{\Mat E}_\infty}
  \le \sqrt{(L-K)(K+1)} \, \epsilon
  \le (L + 1) \, \nicefrac\epsilon2.
\end{equation*}

\begin{theorem}
  \label{the:weight-pert-zeros}
  Let $L > 2K$, and let $\tilde{\Vek \gamma}$ be a normalized right
  singular vector to the smallest singular value $\tilde \sigma_{K}$
  of the perturbed Hankel matrix \eqref{eq:dist-hankel} with respect
  to the exponential sum \eqref{eq:exp-sum}.  Then the corresponding
  polynomial $\tilde P(z) = \sum_{k=0}^K \tilde \gamma_k z^k$
  satisfies
  \begin{equation*}
    \sum_{k=0}^{K-1} \absn{\eta_k}^2 \absn{\tilde P(\beta_k)}^2
    \le  L \, \biggl(\frac{\pi_{\Vek \beta}}{\sigma_{\Vek
        \beta}^{K-1}}\biggr)^2 
    \, \bigl(\tilde \sigma_K + \pNormn{\Mat E}_2\bigr)^2.
  \end{equation*}
\end{theorem}

\begin{proof}
  Let $\tilde{\Vek \nu}$ be the corresponding left singular vector,
  \ie\
  $\tilde{\Mat H} \tilde{\Vek \gamma} = \tilde \sigma_K \tilde{\Vek
    \nu}$.  Incorporating \eqref{eq:dist-hankel} and
  \eqref{eq:exp-sum} into this equation, we obtain
  \begin{equation*}
    \tilde \sigma_K \tilde \nu_\ell
    = \sum_{k=0}^K  \tilde h_{\ell + k} \tilde \gamma_k
    = \sum_{k=0}^K  ( h_{\ell + k} + e_{\ell +k }) \, \tilde \gamma_k
    = \sum_{j=0}^{K-1} \eta_j \beta_j^\ell \, \tilde P(\beta_j)
    + \sum_{k=0}^K e_{\ell + k} \tilde \gamma_k
  \end{equation*}
  for $\ell = 0, \dots, L-K-1$.  In matrix-vector form, these
  equations are given by
  \begin{equation*}
    \Mat V_{L-K}(\Vek \beta) \,
    \Bigl( \eta_j \tilde P(\beta_j) \Bigr)_{j=0}^{K-1}
    = \tilde \sigma_K \tilde{\Vek \nu} - \Mat E \tilde{\Vek \gamma}.
  \end{equation*}
  Multiplying with the left inverse
  $\Mat V_{L-K}^+ (\Vek \beta) \coloneqq \bigl(
  \begin{smallmatrix}
    \Mat V^{-1} (\Vek \beta) \\
    \Mat 0_{L-2K,K}
  \end{smallmatrix}
  \bigr)$,
  we obtain
  \begin{equation*}
    \Bigl( \eta_j \tilde P(\beta_j) \Bigr)_{j=0}^{K-1}
    = \Mat V_{L-K}^+(\Vek \beta) \, ( \tilde \sigma_K \tilde{\Vek \nu}
    - \Mat E \tilde{\Vek \gamma}). 
  \end{equation*}
  Taking the squared Euclidean norm, bounding the spectral norm by the
  row-sum norm, and applying \thref{prop:norm-inv-vander} yields the
  assertion.  \qed
\end{proof}

\begin{theorem}
  \label{the:pert-zeros}
  Let $L > 2K$, let $\tilde{\Vek \gamma}$ be a normalized right
  singular vector to the smallest singular value $\tilde \sigma_{K}$
  of the perturbed Hankel matrix \eqref{eq:dist-hankel} with
  $\absn{h_\ell - \tilde h_\ell} \le \epsilon$, and let $\sigma_{K-1}$
  be the smallest non-zero singular value of the unperturbed Hankel
  matrix \eqref{eq:hankel}. Then the corresponding polynomial
  $\tilde P(z) = \sum_{k=0}^K \tilde \gamma_k z^k$ satisfies
  \begin{equation*}
    \sum_{k=0}^{K-1} \absn{\tilde P(\beta_k)}^2
    \le  KL \, \rho_{\Vek \beta}^{2L-2}
    \; \frac{\bigl(\tilde{\sigma}_K
      + \pNormn{\Mat E}_2\bigr)^2}{\sigma_{K-1}^2} 
  \end{equation*}
\end{theorem}

\begin{proof}
  First assume $\tilde{\Vek \gamma} \notin \ker \Mat H$.  Letting
  $\Vek \gamma \coloneqq \proj_{\ker \Mat H} \tilde{\Vek \gamma}$, the
  projection $\Vek \gamma$ is a maybe not normalized right singular
  vector for the singular value zero.
  \thref{lem:singular-value-and-roots} implies that the polynomial
  $P(z) \coloneqq \sum_{k=0}^{K} \gamma_k z^k$ has the roots
  $\beta_0, \dots,\beta_{K-1}$. Therefore, we can write
  \begin{equation*}
    \sum_{k=0}^{K-1}  | \tilde{P} (\beta_k) |^2
    = \sum_{k=0}^{K-1}  | \tilde{P}(\beta_k) - P (\beta_k)|^2
    = \| \Mat V^\T(\Vek\beta) \, \tilde{\Vek \gamma} -
    \Mat V^\T(\Vek\beta) \, \Vek \gamma \|_2^2
    \leq \| \Mat V (\Vek\beta) \|_2^2 \,
    \| \tilde{\Vek \gamma} - \Vek \gamma \|^2_2. 
  \end{equation*}
  Now since
  $(\tilde{\Vek \gamma} - \Vek \gamma) \perp \ker \Mat H$, we
  obtain
  \begin{equation*}
    \sigma^2_{K-1} \| \tilde{\Vek \gamma} - \Vek \gamma \|^2_2
    \leq \| {\Mat H}(\tilde{\Vek \gamma} - \Vek \gamma)\|^2_2
    = \|(\tilde{\Mat H} - \Mat E) \tilde{\Vek \gamma}\|^2_2
    \le \bigl(\tilde \sigma_{K} + \pNormn{\Mat E}_2 \bigr)^2.
  \end{equation*}
  Combining the above inequalities, and applying
  \thref{lem:norm-vander}, we establish the assertion.  For the
  remaining case $\tilde{\Vek \gamma} \in \ker \Mat H$, the bases
  $\beta_k$ are roots of $\tilde P$ by
  \thref{lem:singular-value-and-roots}.  \qed
\end{proof}

\begin{remark}
  The above Theorems \ref{the:weight-pert-zeros} and
  \ref{the:pert-zeros} essentially state that the true bases are
  nearly roots of the perturbed Prony polynomial.  Therefore, we
  nurture the hope that the perturbed roots are close.  Although this
  seems plausible for generic polynomials, we can construct
  pathological cases of very sensitive polynomials, where already
  slight disturbances of the coefficients have tremendous effects on
  the roots.  In \cite{tang2017}, the author tries to establishes an
  explicit bound on the reconstruction error regarding the roots of
  the Prony polynomial, which we initially wanted to adapt to our
  setting.  Unfortunately, the key theorem studying a linear
  perturbation of the coefficient of a polynomial cannot be applied to
  our setting since here the perturbations $e_\ell$ in the
  measurements $\tilde h_\ell = h_\ell + e_\ell$ lead to non-linear
  perturbations of the coefficients in the Prony
  polynomial.  \qed
\end{remark}

In the third step of Prony's method, the coefficients $\Vek \eta$ of
the exponential sum \eqref{eq:exp-sum} are determined by solving
$\Mat V_{L}(\Vek \beta) \, \Vek \eta = \tilde{\Vek h}$ in the
least-square sense, \ie\ we have to determine the minimizer of
$\pNormn{\Mat V_{L}(\Vek \beta) \, \Vek \eta - \tilde{\Vek h}}_2$.
The minimizer is given by
$\Mat V_{L}^\dagger (\Vek \beta) \, \tilde{\Vek h}$, where
\begin{equation*}
  \Mat V_{L}^\dagger(\Vek \beta)
  = (\Mat V_{L}^*(\Vek \beta) \, \Mat V_{L}(\Vek
  \beta))^{-1} \, \Mat V_{L}^*(\Vek \beta)
\end{equation*}
is the Moore--Penrose inverse.  To estimate the reconstruction error
with respect to $\Vek \eta$, we need to estimate the norm of the
Moore--Penrose inverse.  For this, we exploit that the Moore--Penrose
inverse is the zero continuation of the inverse with respect to the
range of the orthogonal complement of the kernel.  For an arbitrary
full-rank matrix, the Moore--Penrose inverse is therefore the left
inverse with the smallest norm.

\begin{proposition}
  \label{prop:moore-penrose}
  Let $\Mat A \in \BC^{L \times K}$ with $L \ge K$ be a full-rank
  matrix, and let $\Mat A^+$ be an arbitrary left inverse.
  For every $1 \le p \le \infty$, the Moore--Penrose inverse then satisfies
  \begin{equation*}
    \pNormn{\Mat A^\dagger}_p \le \pNormn{\Mat A^+}_p.
  \end{equation*}
\end{proposition}

\begin{proof}
  Since every left inverse $\Mat A^+$ fulfils
  $\Mat A^+ \Mat A = \Mat I$, all left inverses coincide on the range
  of $\Mat A$.  The Moore--Penrose inverse is now the unique zero
  continuation from the range to the whole space $\BC^L$, which
  geometrically means that the Moore--Penrose inverse is the
  projection onto $\ran \Mat A$ composed with the unique inverse on
  the range.  For the induced matrix norm, this means
  \begin{equation*}
    \pNormn{\Mat A^+}_p
    = \sup_{\pNormn{\Vek x}_p = 1} \pNormn{\Mat A^+ \Vek x}_p
    \ge \sup_{\substack{\pNormn{\Vek x}_p = 1\\ \Vek x \in \ran \Mat
        A}}
    \pNormn{\Mat A^+ \Vek x}_p
    = \sup_{\pNormn{\Vek x}_p = 1} \pNormn{\Mat A^\dagger \Vek x}_p
    = \pNormn{\Mat A^\dagger}_p
  \end{equation*}
  because $(\ran \Mat A)^\perp = \ker \Mat A^\dagger$.  This
  argumentation holds for all induced matrix norms and not only for
  the $p$-norm.  \qed
\end{proof}

Using this property of the Moore--Penrose inverse, we may immediately
estimate the condition number
$\kappa(\Mat V_{L}(\Vek \beta)) \coloneqq \pNormn{\Mat
  V_{L}^\dagger(\Vek \beta)}_2 \pNormn{\Mat V_{L}(\Vek \beta)}_2$ of
the Vandermonde matrix $\Mat V_{L}(\Vek \beta)$ if the bases
$\Vek \beta$ are known.

\begin{proposition}
  The condition number of the Vandermonde matrix
  $\Mat V_{L}(\Vek \beta)$ is bounded by
  \begin{equation*}
    \kappa(\Mat V_{L}(\Vek \beta))
    \le \sqrt K \, L \,
    \frac{\pi_{\Vek \beta} \, \rho_{\Vek \beta}^{L-1}}
    {\sigma_{\Vek \beta}^{K-1}},
  \end{equation*}
\end{proposition}

\begin{proof}
  The bound follows from \thref{lem:norm-vander} and from
  \thref{prop:moore-penrose} and \ref{prop:norm-inv-vander} with the
  left inverse $\Mat V_{L-K}^+ (\Vek \beta) \coloneqq \bigl(
  \begin{smallmatrix}
    \Mat V^{-1} (\Vek \beta) \\
    \Mat 0_{L-2K,K}
  \end{smallmatrix}
  \bigr)$.  \qed
\end{proof}

\begin{proposition}
  Let $\Vek \eta$ and $\Vek \beta$ be the parameters of the
  exponential sum \eqref{eq:exp-sum}.  The least-squares solution
  $\tilde{\Vek \eta}$ of the perturbed equation system
  $\Mat V_{L}(\Vek \beta) \, \tilde{\Vek \eta} = \tilde{\Vek h}$ with
  $\pNormn{\Vek h - \tilde{\Vek h}}_\infty \le \epsilon$ satisfies
  \begin{equation*}
    \pNormn{\Vek \eta - \tilde{\Vek \eta}}_\infty
    \le \frac{\pi_{\Vek \beta}}{\sigma_{\Vek \beta}^{K-1}} \, \epsilon.
  \end{equation*}
\end{proposition}

\begin{proof}
  The inequality follows immediately from
  $\pNormn{\Vek \eta - \tilde{\Vek \eta}}_\infty \le \pNormn{\Mat
    V_{L}^\dagger(\Vek \beta)}_\infty \pNormn{\Vek h - \tilde{\Vek
      h}}_\infty$ and from applying \thref{prop:moore-penrose} and
  \ref{prop:norm-inv-vander} with the left inverse
  $\Mat V_{L-K}^+ (\Vek \beta) \coloneqq \bigl(
  \begin{smallmatrix}
    \Mat V^{-1} (\Vek \beta) \\
    \Mat 0_{L-2K,K}
  \end{smallmatrix}
  \bigr)$.  \qed
\end{proof}

Certainly, the computed bases $\tilde{\Vek \beta}$ are themselves only
approximations of $\Vek \beta$ in practice.  Therefore, besides the
right-hand side $\tilde{\Vek h}$, the Vandermonde matrix $\Mat
V_{L}(\tilde{\Vek \beta})$ is perturbed too.  For studying the effect
to the recovered coefficients, we need the following lemmata.

\begin{lemma}
  \label{lem:pert-prod-rad}
  For $\Vek \beta \in \BC^{K}$, and for $\tilde{\Vek \beta} \in
  \BC^K$ with $\pNormn{\Vek \beta - \tilde{\Vek \beta}}_\infty \le
  \delta$, it holds
  \begin{equation*}
    \pi_{\tilde{\Vek \beta}} \le \pi_{\absn{\Vek \beta} + \delta \Vek 1}.
  \end{equation*}
\end{lemma}

\begin{proof}
  The lemma is established by
  \begin{equation*}
    \pi_{\tilde{\Vek \beta}}
    = \prod_{k=0}^{K-1} (1 + \absn{\tilde \beta_k})
    \le \prod_{k=0}^{K-1} (1 + \absn{\beta_k} + \delta)
    = \pi_{\absn{\Vek \beta} + \delta \Vek 1}.
    \tag*{\qed}
  \end{equation*}
\end{proof}

\begin{lemma}
  \label{lem:pert-min-sep}
  For $\Vek \beta \in \BC^{K}$, and for
  $\tilde{\Vek \beta} \in \BC^K$ with
  $\pNormn{\Vek \beta - \tilde{\Vek \beta}}_\infty \le \delta$, it
  holds
  \begin{equation*}
    \sigma_{\tilde{\Vek \beta}} \ge \sigma_{\Vek \beta} - 2 \delta.
  \end{equation*}
\end{lemma}

\begin{proof}
  Using the triangle inequality, we may estimate the minimal
  separation by
  \begin{equation*}
    \absn{\tilde \beta_\ell - \tilde \beta_k}
    \ge \absn{\beta_\ell - \beta_k} - \absn{\beta_\ell - \tilde
      \beta_\ell} - \absn{\beta_k - \tilde \beta_k}
    \ge \absn{\beta_\ell - \beta_k} - 2 \delta.
    \tag*{\qed}
  \end{equation*}
\end{proof}

\begin{lemma}
  \label{lem:per-vander}
  For $\Vek \beta \in \BC^{K}$, and for
  $\tilde{\Vek \beta} \in \BC^K$ with
  $\pNormn{\Vek \beta - \tilde{\Vek \beta}}_\infty \le \delta$, it
  holds
  \begin{equation*}
    \pNormn{\Mat V_{L}(\tilde{\Vek \beta}) - \Mat V_{L}(\Vek
      \beta)}_\infty
    \le \sqrt 2 \, KL \, \rho_{\absn{\Vek \beta} + \delta \Vek 1}^{L-1} \, \delta.
  \end{equation*}
\end{lemma}

\begin{proof}
  We use the following complex mean value theorem
  \cite[Thm~2.2]{EJ92}: Let $f$ be a holomorphic function defined on
  an open convex set $D \subset \BC$, and let $a$ and $b$ be two
  distinct points in $D$.  Then there exist $z_1$, $z_2 \in (a,b)$
  such that
  \begin{equation*}
    \Re \bigl( f'(z_1) \bigr) = \Re \biggl( \frac{f(b) - f(a)}{b-a} \biggr)
    \qquad\text{and}\qquad
    \Im \bigl( f'(z_2) \bigr) = \Im \biggl( \frac{f(b) - f(a)}{b-a} \biggr),
  \end{equation*}
  where $(a,b)$ denotes the open line segment
  \begin{equation*}
    (a,b) \coloneqq \{a + t (b-a) : t \in (0,1) \}.
  \end{equation*}
  On the basis of this complex mean value theorem, we obtain
  \begin{equation*}
    \pNormn{\Mat V_{L}(\Vek \beta) - \Mat V_{L}(\tilde{\Vek \beta})}_\infty
    = \max_{0 \le \ell < L} \sum_{k=0}^{K-1} \absn{ \beta_k^\ell - \tilde{\beta}_k^\ell } 
    = \max_{0 \le \ell < L} \sum_{k=0}^{K-1} \ell \absn{\beta_k -
      \tilde \beta_k} \absn{\Re(\xi_{\ell,k}^{\ell-1}) + \I \Im (\zeta_{\ell,k}^{\ell-1})}
  \end{equation*}
  with intermediate points $\xi_{\ell,k}$, $\zeta_{\ell,k} \in
  (\beta_k, \tilde{\beta}_k)$.  Since $\absn{\xi_{\ell,k}} \le
  \absn{\beta_k} + \delta$ as well as  $\absn{\zeta_{\ell,k}} \le
  \absn{\beta_k} + \delta$, we finally have
  \begin{equation*}
     \pNormn{\Mat V_{L}(\Vek \beta) - \Mat V_{L}(\tilde{\Vek
    \beta})}_\infty
    \le \max_{\substack{0 \le \ell < L \\ 0 \le k < K}}\sqrt 2 \, \ell K \,
    (\absn{\beta_k} + \delta)^{\ell-1} \, \delta
    \le \sqrt 2 \, KL \, \rho_{\absn{\Vek \beta} + \delta \Vek 1}^{L-1} \,
    \delta.
    \tag*{\qed}
  \end{equation*}
\end{proof}

\begin{theorem}
  Let $\Vek \eta$ and $\Vek \beta$ be the parameters of the
  exponential sum \eqref{eq:exp-sum}.  The least-squares solution
  $\tilde{\Vek \eta}$ of the perturbed equation system
  $\Mat V_{L}(\tilde{\Vek \beta}) \, \tilde{\Vek \eta} = \tilde{\Vek
    h}$ with $\pNormn{\Vek h - \tilde{\Vek h}}_\infty \le \epsilon$,
  $\pNormn{\Vek \beta - \tilde{\Vek \beta}}_\infty \le \delta$, and
  $\delta < \nicefrac{\sigma_{\Vek \beta}}2$ satisfies
  \begin{equation*}
    \pNormn{\Vek \eta - \tilde{\Vek \eta}}_\infty
    \le
    \frac{\pi_{\absn{\Vek \beta} + \delta \Vek 1}}{(\sigma_{\Vek
        \beta} - 2 \delta)^{K-1}} \,
    \biggl(\sqrt 2 \,  KL 
    \, \frac{\pi_{\Vek \beta}  \, \rho_{\absn{\Vek \beta} + \delta
        \Vek 1}^{L-1}}
    {\sigma_{\Vek \beta}^{K-1}} \pNormn{\Vek h}_\infty \,
    \delta + \epsilon \biggr).
  \end{equation*}
\end{theorem}

\begin{proof}
  Due to $\delta < \nicefrac{\sigma_{\Vek \beta}}2$, the perturbed
  Vandermonde matrix $\Mat V_{L}(\tilde{\Vek \beta})$ has full rank.
  Further, the reconstruction error may be estimated by
  \begin{align*}
    \pNormn{ \Vek\eta - \tilde{\Vek\eta} }_\infty
    &= \pNormn{\Vek \eta - \Mat V_{L}^\dagger(\tilde{\Vek \beta}) \,
      \tilde{\Vek h}}_\infty
    \\[\fskip]
    &= \pNormn{\Mat V_{L}^\dagger(\tilde{\Vek \beta})\, \Mat
      V_{L}(\tilde{\Vek \beta}) \Vek \eta
      - \Mat V_{L}^\dagger(\tilde{\Vek \beta}) \, \Mat V_{L}({\Vek
      \beta}) \, \Vek \eta
      + \Mat V_{L}^\dagger(\tilde{\Vek \beta}) (\Vek h - \tilde{\Vek h}) }_\infty
    \\[\fskip]
    &\leq 
      \pNormn{\Mat V_{L}^\dagger(\tilde{\Vek \beta})}_\infty
      \bigl(\pNormn{\Mat V_{L}(\tilde{\Vek \beta}) - \Mat V_{L}({\Vek
      \beta})}_\infty \pNormn{\Vek \eta}_\infty
      + \pNormn{\Vek h - \tilde{\Vek h}}_\infty \bigr)
  \end{align*}
  The first factor may be estimated by applying
  \thref{prop:moore-penrose} with perturbed left inverse
  $\Mat V_{L-K}^+ (\tilde{\Vek \beta}) \coloneqq \bigl(
  \begin{smallmatrix}
    \Mat V^{-1} (\tilde{\Vek \beta}) \\
    \Mat 0_{L-2K,K}
  \end{smallmatrix}
  \bigr)$ followed by \thref{prop:norm-inv-vander},
  \thref{lem:pert-prod-rad}, and \thref{lem:pert-min-sep} yielding
  \begin{equation*}
    \pNormn{\Mat V_{L}^\dagger(\tilde{\Vek \beta})}_\infty
    \le \frac{\pi_{\tilde{\Vek \beta}}}{\sigma_{\tilde{\Vek
          \beta}}^{K-1}}
    \le \frac{\pi_{\absn{\Vek \beta} + \delta \Vek 1}}{(\sigma_{\Vek
        \beta} - 2 \delta)^{K-1}}.
  \end{equation*}
  Using \thref{lem:per-vander} and that $\pNormn{\Vek \eta}_\infty \le
  \pNormn{\Mat V_{L}^\dagger(\Vek \beta)}_\infty \pNormn{\Vek h}_\infty$
  together with \thref{prop:moore-penrose} and
  \thref{prop:norm-inv-vander}, we finally arrive at
  \begin{equation*}
    \pNormn{\Vek\eta - \tilde{\Vek\eta}}_\infty
    \le
    \frac{\pi_{\absn{\Vek \beta} + \delta \Vek 1}}{(\sigma_{\Vek
        \beta} - 2 \delta)^{K-1}} \,
    \biggl(\sqrt2 \,  KL \, \rho_{\absn{\Vek \beta} + \delta \Vek 1}^{L-1} \,
    \delta
    \, \frac{\pi_{\Vek \beta}}{\sigma_{\Vek \beta}^{K-1}} \pNormn{\Vek
      h}_\infty + \epsilon \biggr).
    \tag*{\qed}
  \end{equation*}
\end{proof}

\subsection{Sensitivity of phase \& system identification}
\label{sec:sens-phase-system}

On the basis of the sensitivity analysis of Prony's method, we analyse
the error propagation in dynamical phase retrieval.  For this, we
assume that the unknown bases $\lambda_j \bar\lambda_k$ and
coefficients $c_j \bar c_k$ of the exponential sum describing the
measurements \eqref{eq:phaseless-meas} have been approximately computed.
In the following, we denote the true bases and coefficients by
\begin{align}\label{eq:main-structure}
  {\beta}_{\tau(j,k)} ={\lambda}_j \bar{{\lambda}}_k
  \qquad\text{and}\qquad
  \eta_{\tau(j,k)} = c_j \bar c_k,
\end{align}
where the bijective map
\begin{equation*}
  \tau\colon \{0,\dots,d-1\} \times \{0,\dots,d-1\}
  \to \{0, \dots, d^2-1\}
\end{equation*}
describes the relation between the indices.  Assuming that the
recovered bases $\tilde{\Vek \beta}$ and coefficients
$\tilde{\Vek \eta}$ satisfy
$\|\Vek{\tilde{\beta}}- \Vek \beta \|_\infty \leq \delta$ and
$\|\Vek{\tilde{\eta}}- \Vek \eta \|_\infty \leq \varepsilon$, where
$\delta$ should be small enough such that the mapping $\tau$ can be
recovered up to the winding direction by the above constructive
proofs, \ie\ the error is small enough such that the order of the
absolute values $\absn{\beta_k}$ remains unchanged, we want to
estimate the errors in the recovered spectrum $\tilde{\Vek\lambda}$
and signal $\tilde{\Vek x}$.  Note that $\tilde\beta_{\tau(j,k)}$ and
$\tilde\eta_{\tau(j,k)}$ are simply conjugated for the opposite
winding direction.

In line with the above procedures, where firstly the magnitudes of the
unknown variables are determined, and secondly the phase is propagated
between the elements, we decouple the sensitivity analysis of absolute value and
phase.  Further, we first discuss the sensitivity of the unknown
operator spectrum, followed by the analysis of the unknown signal, and
finally the error propagation for multiple sampling vectors.

\paragraph{Sensitivity of the spectrum}

The recovered bases $\tilde{\Vek\beta}$ already contain estimates of
the squared modulus of the spectrum $\Vek \lambda$.  After recovering
the relation $\tau$ (up to winding direction), the magnitude of the
spectrum is easily obtained by taking the square root, \ie\
\begin{equation}
  \label{eq:mag-spectrum}
  \absn{\tilde\lambda_j} \coloneqq \sqrt{\absn{\tilde \beta_{\tau(j,j)}}}.
  \addmathskip
\end{equation}
The sensitivity of the magnitude computation may be easily estimated
via the mean value theorem.

\begin{lemma}\label{lem:error_in_each_amplitude}
  Assume  $\absn{ \tilde{\beta}_{\tau(j,j)}  -  {\beta}_{\tau(j,j)}}
  \leq \delta$, and estimate the magnitude $\absn{\lambda_j}$ by
  \eqref{eq:mag-spectrum}.  If $\delta < \absn{\lambda_j}^2$, then we have  
  \begin{equation*}
    \absb{\absn{\tilde{\lambda}_j} - \absn{\lambda_j}}
    \leq \frac{\delta}{2\sqrt{\absn{\lambda_j}^2 - \delta}}
  \end{equation*}
  and, for $\delta < \nicefrac{\absn{\lambda_j}^2}2 $, in particular 
  \begin{equation*}
    \absb{ \absn{\tilde{\lambda}_j} - \absn{\lambda_j}} \leq
    \frac{\sqrt{2} \, \delta}{2\sqrt{\absn{\lambda_j}^2}}.
  \end{equation*}
\end{lemma}

\begin{proof}
  The statement immediately follows from applying the mean value
  theorem and the reversed triangle inequality by
  \begin{equation*}
    \absb{\absn{\tilde{\lambda}_j} - \absn{\lambda_j}}
    = \absb{\absn{\tilde{\beta}_{\tau(j,j)}}^{\nicefrac12} - \absn{\beta_{\tau(j,j)} }^{\nicefrac12} } 
    \leq \frac{ \delta}{2\sqrt{\absn{\beta_{\tau(j,j)}}  - \delta}}.
  \end{equation*}
  The second one is a trivial consequence.  \qed
\end{proof}

Recall that for computing the phase of $\tilde{\lambda}_j$, we first
find the element with the largest magnitude, say $\tilde\lambda_k$,
then set the phase of $\tilde\lambda_k$ to be zero due to the global
phase ambiguity, and finally propagate the phase to
$\tilde{\lambda}_j$ using the relative phase encoded in
$\beta_{\tau(j,k)}$.  More precisely, exploiting
$\tilde{\beta}_{\tau(j,k)} \approx {\lambda}_j \bar{{\lambda}}_k$, we
retrieve the phase of ${\lambda}_j$ by
\begin{align}
  \label{eq:phase-spectrum}
  \tilde{\lambda}_j \coloneqq
  \frac{\tilde{\beta}_{\tau(j,k)}}{|\tilde{\beta}_{\tau(j,k)}|}
  \, \absn{\tilde\lambda_j},
\end{align}
where $\absn{\tilde\lambda_j}$ has been computed by
\eqref{eq:mag-spectrum} in the first step.  Note that this phase
propagation is a very simple method, which however allow to analyse
the propagation error. For doing this, we
assume that the map $\tau$ given in \eqref{eq:main-structure} has been
identified with respect to the true winding direction.  Otherwise, we
consider the conjugated recovered spectrum
\raisebox{0pt}[0pt][0pt]{$\bar{\tilde{\Vek \lambda}}$} without loss of
generality.  For simplicity, we first consider the phase propagation
only between two elements. The idea of the proof was motivated by \cite{iwen2016fast}.

\begin{lemma}\label{lem:error_in_each_phase}
  Assume
  $\absn{ \tilde{\beta}_{\tau(j,k)} - {\beta}_{\tau(j,k)}} \leq
  \delta$, suppose that $\lambda_k$ is real and positive, and estimate
  the phase $\arg(\lambda_j)$ by \eqref{eq:phase-spectrum}.  If
  $\delta < \absn{\lambda_j} \absn{\lambda_k}$, then we have
  \begin{equation*}
    |\arg(\tilde{\lambda}_j ) - \arg(\lambda_j)\Mod2\uppi|
    \leq \frac{2\delta}{\absn{\lambda_k}\absn{\lambda_j}}.
  \end{equation*}
\end{lemma}

\begin{proof}
  Since $\lambda_k$ is supposed to be real and positive, the phase of
  $\lambda_j$ is directly encoded in the basis $\beta_{\tau(j,k)}$ by
  \begin{equation*}
    \arg(\beta_{\tau(j,k)})
    = \arg(\lambda_j)- \arg(\lambda_k) \Mod 2 \uppi
    = \arg({\lambda}_j).
  \end{equation*}
  During the proof, we denote the phases of
  $\beta_{\tau(j,k)}$ and $\tilde\beta_{\tau(j,k)}$ or $\lambda_j$ and
  $\tilde \lambda_j$ by $\alpha_j$ and $\tilde\alpha_j$ respectively.
  Because of
  $\absn{\tilde{\beta}_{\tau(j,k)}-\beta_{\tau(j,k)}} \leq \delta <
  \absn{ \beta_{\tau(j,k)}}$, the phase difference
  $|\tilde \alpha_j - \alpha_j \Mod2 \uppi|$ is always smaller than
  $\nicefrac{\uppi}2$.  Thus we have
  \begin{equation*}
    \absn{\tilde\alpha_j - \alpha_j \Mod 2\uppi}
    \le 2 \sin(\absn{\tilde\alpha_j - \alpha_j \Mod 2\uppi}).
  \end{equation*}
  To estimate the sine of the phase difference, we exploit the
  geometrical relation between $\beta_{\tau(j,k)}$ and $\tilde
  \beta_{\tau(j,k)}$ schematically presented in
  \autoref{fig:proof-err-phase}.  Using the best-known sine relation
  of the right-angled triangle, we have
  \begin{equation*}
    \absn{\tilde\alpha_j - \alpha_j \Mod 2\uppi}
    \le \frac{2 \gamma}{\absn{\beta_{\tau(j,k)}}}
    \le \frac{2 \delta}{|\beta_{\tau(j,k)}|}.
    \tag*{\qed}
  \end{equation*}
\end{proof}

\begin{figure}
  \centering
  \includegraphics{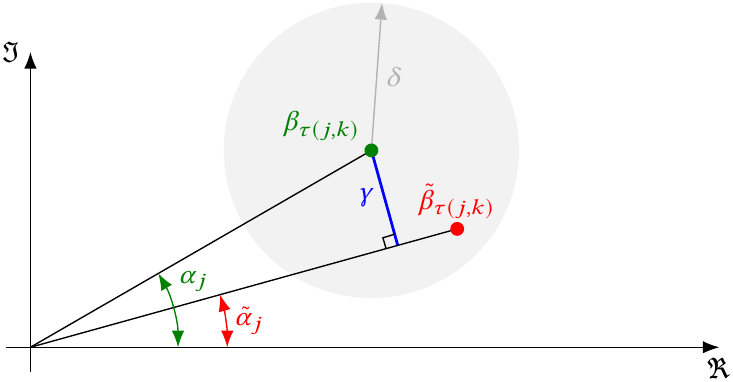}
  \caption{Geometrical relation between $\beta_{\tau(j,k)}$ and
    $\beta_{\tau(j,k)}$.  In the proof of
    Lemma~\ref{lem:error_in_each_phase}, we exploit the right-angled
    triangle between the rays with angle $\alpha_j$ and
    $\tilde\alpha_j$.  Note that the point $\tilde\beta_{\tau(j,k)}$
    may lay on the adjacent.  The opposite $\gamma$ is of length
    $\delta$ at the most.}
  \label{fig:proof-err-phase}
\end{figure}

Coupling the recovery of absolute values and the phase, we may
estimate the total recovery error for the spectrum $\Vek\lambda$,
which mainly depends on $\pNormn{\Vek \lambda}_{-\infty}$.

\begin{proposition}\label{prop:total-error}
  Assume
  $\pNormn{ \tilde{\Vek\beta} - {\Vek \beta}}_\infty \leq \delta$, and
  estimate $\Vek \lambda$ by \eqref{eq:mag-spectrum} and
  \eqref{eq:phase-spectrum}, where the true winding direction is used
  without loss of generality, and where the phase is propagated from
  the element largest in magnitude.  If
  $\delta < \pNormn{\Vek\lambda}^2_{-\infty}$, then we have
  \begin{equation*}
    \|\tilde{\Vek\lambda} -  \Vek\lambda \|_{\infty}
    \leq \biggl( \frac{2\sqrt{2}}{ \|\Vek\lambda\|_{-\infty} } +
    \frac{1}{2\sqrt{\|\Vek\lambda\|_{-\infty}^2-\delta}}\biggr) \, \delta
  \end{equation*}
  and, for $\delta \leq \nicefrac{\pNormn{\Vek\lambda}^2_{-\infty}}2$,
  in particular
  \begin{equation*}
    \|\tilde{\Vek\lambda} -  \Vek\lambda \|_{\infty}
    \leq  \frac{5 \sqrt{2} \, \delta}{2\|\Vek\lambda\|_{-\infty}}.
  \end{equation*}
\end{proposition}

\begin{proof}
  Let $\tilde{\alpha}_j,\alpha_j$ be the phases of
  $\tilde{\lambda}_j,\lambda_j$ respectively.  We decouple the phase
  and magnitude error by 
  \begin{equation*}
    \absn{\tilde{\lambda}_j - \lambda_j}
    = \absb{ \absn{\tilde{\lambda}_j} \, \e^{\I\tilde{\alpha}_j} 
      \pm  \absn{\tilde{\lambda}_j } \, \e^{\I\alpha_j}
      - \absn{\lambda_j} \e^{\I\alpha_j} }
      \leq \absn{\tilde{\lambda}_j} \,
      \absn{\e^{\I\tilde{\alpha}_j} - \e^{\I\alpha_j}}
      + \absb{\absn{\tilde{\lambda}_j} - \absn{\lambda_j}}.
  \end{equation*}
  The magnitude error may be simply estimated using
  \thref{lem:error_in_each_amplitude} via 
  \begin{equation*}
    \absb{\absn{\tilde{\lambda}_j} - \absn{\lambda_j}}
    \leq  \frac{\delta}{2\sqrt{\|\Vek\lambda\|_{-\infty}^2-\delta}}.
  \end{equation*}
  For the phase error, assume that
  $\lambda_k$ is the eigenvalue with largest magnitude, set
  $\arg(\lambda_k) = 0$, and propagate the phase from $\lambda_k$ to
  the remaining $\lambda_j$ by \eqref{eq:phase-spectrum}. The
  difference between the unimodular exponentials is now
  \begin{align*}
    | \e^{\I\tilde{\alpha}_j} - \e^{\I\alpha_j} |
    &=| \e^{\nicefrac{\I(\tilde{\alpha}_j -\alpha_j)}2}
      -  \e^{\nicefrac{-\I(\tilde{\alpha}_j -\alpha_j)}2}|
      =  2|\sin(\nicefrac{(\tilde{\alpha}_j - \alpha_j)}{2})|
    \\
    & \leq | \tilde{\alpha}_j - \alpha_j \Mod 2 \uppi |	 
      \le\frac{2\delta}{\|\Vek\lambda\|_{\infty} \|\Vek\lambda\|_{-\infty} },
  \end{align*}
  where the last inequality holds by \thref{lem:error_in_each_phase}.
  Using
  $|\tilde{\lambda_j}|\leq \sqrt{\|\Vek\lambda\|^2_{\infty}+ \delta}
  \le \sqrt 2 \pNormn{\Vek \lambda}_\infty$, we finally arrive at
  \begin{equation*}
    \| \Vek{\tilde{\lambda}} - \Vek\lambda\|_\infty 
    \leq \frac{2\delta\,\sqrt{\|\Vek \lambda\|^2_{\infty}
        +\delta}}{\|\Vek\lambda\|_{\infty} \|\Vek\lambda\|_{-\infty} }
    +  \frac{\delta}{2\sqrt{\|\Vek\lambda\|_{-\infty}^2-\delta}}
    \leq \biggl( \frac{2\sqrt{2}}{ \|\Vek\lambda\|_{-\infty} } +
    \frac{1}{2\sqrt{\|\Vek\lambda\|_{-\infty}^2-\delta}} \biggr) \,\delta.
  \end{equation*}
  If $\delta< \nicefrac{\|\Vek\lambda\|_{-\infty}}2$, we obtain
  \begin{equation*}
    \| \Vek{\tilde{\lambda}} - \Vek\lambda\|_\infty \leq  
    \frac{2\sqrt{2} \, \delta}{ \|\Vek\lambda\|_{-\infty} }
    +  \frac{\sqrt{2}\,\delta}{2\|\Vek\lambda\|_{-\infty}}
    \leq \frac{5\sqrt{2}\delta}{2\|\Vek\lambda\|_{-\infty}}.
    \tag*{\qed}
  \end{equation*}
\end{proof}

\paragraph{Sensitivity of the signal}

As discussed in the previous sections, the components of $\tilde{\Vek\eta}$
are in line with the structure of \eqref{eq:main-structure} meaning
\begin{equation*}
  \tilde{\eta}_{\tau(j,k)} \approx {c}_j \bar{{c}}_k
  \qquad\text{with}\qquad  
  {c}_j = \bar{y}_j\psi_j
  = (\overline{\Mat S^*{\Vek x}})_j \, (\Mat S^{-1}\Vek\phi)_j.
\end{equation*}
With respect to the above proofs, we recover the transformed signal
$\Vek y = \Mat S^* \Vek x$ similar to the spectrum $\Vek \lambda$.
Thus, we first recover the magnitudes via the real and positive values
$\tilde{\eta}_{\tau(j,j)}$, then assume that $\tilde c_k$ largest in
magnitude is real and positive, and spread the phase from $\tilde c_k$
to every other $\tilde c_j$ using the relative phase encoded in
$\tilde \eta_{\tau(j,k)}$.  Because of $y_j = c_j \psi_j^{-1}$
resulting in $|y_j| = |\psi_j^{-1}| \, \sqrt{\eta_{\tau(j,j)}}$ and
$\arg(y_j) = \arg(\eta_{\tau(j,k)}) - \arg(\psi_j)$, we compute the
tranformed components via
\begin{align} \label{eq:coeff-recover}
  \absn{\tilde y_j} \coloneqq
  \frac{\sqrt{\absn{\tilde\eta_{\tau(j,j)}}}}{\absn{\psi_j}}
  \qquad\text{and}\qquad
  y_j \coloneqq
  \frac{\eta_{\tau(j,k)}}{|\eta_{\tau(j,k)}|}
  \, \frac{\bar{\psi_j}}
  {|{\psi}_j|} \, \absn{y_j}.
\end{align}
Adapting the considerations in the previous paragraph for the
spectrum, we obtain the following sensitivities.

\begin{lemma}\label{lem:co-rec-modul}
  Assume  $\absn{ \tilde{\eta}_{\tau(j,j)}  -  {\eta}_{\tau(j,j)}}
  \leq \epsilon$, and estimate the magnitude $\absn{y_j}$ by
  \eqref{eq:coeff-recover}.  If $\epsilon < \absn{y_j}^2 \absn{\psi_j}^2$, then we have  
  \begin{equation*}
    \absb{\absn{\tilde y_j} - \absn{y_j}}
    \leq \frac{\epsilon}{2 \absn{\psi_j} \sqrt{\absn{y_j}^2 \absn{\psi_j}^2 - \epsilon}}.
  \end{equation*}
\end{lemma}

\begin{proof} 
  Consider
  $| |\tilde{y}_j| - |y_j || =|\psi_j^{-1}| \, ||\tilde{c}_j |- |c_j
  ||$ and use the arguments in \thref{lem:error_in_each_amplitude}. \qed
\end{proof}

\begin{lemma}\label{lem:co-rec-phase}
  Assume
  $\absn{ \tilde{\eta}_{\tau(j,k)} - {\eta}_{\tau(j,k)}} \leq
  \epsilon$, suppose that $y_k$ is real and positive, and estimate
  the phase $\arg(y_j)$ by \eqref{eq:coeff-recover}.  If
  $\epsilon < \absn{y_j} \absn{y_k} \absn{\psi_j} \absn{\psi_k}$, then we have
  \begin{equation*}
    |\arg(\tilde{y}_j ) - \arg(y_j)\Mod2\uppi|
    \leq
    \frac{2\epsilon}{\absn{y_k}\absn{y_j}\absn{\psi_k}\absn{\psi_j}}. 
  \end{equation*}
\end{lemma}

\begin{proof}
  Note that the phase difference may be written as 
  \begin{equation*}
    | \arg(\tilde{y}_j) - \arg({y}_j) |
    = | \arg(\tilde{c}_j) - \arg(\psi_j) - \arg({c}_j) + \arg(\psi_j)|
    = | \arg(\tilde{c}_j) - \arg({c}_j)|,
  \end{equation*}
  and use the arguments of \thref{lem:co-rec-phase}. \qed
\end{proof}

\begin{proposition}
  Assume
  $\pNormn{ \tilde{\Vek\eta} - {\Vek \eta}}_\infty \leq \epsilon$, and
  estimate $\Vek y$ by \eqref{eq:coeff-recover}, where the true
  winding direction is used without loss of generality, and where the
  phase is propagated form the element largest in magnitude.  If
  $\epsilon < \pNormn{\Vek y}^2_{-\infty} \pNormn{\Vek
    \psi}_{-\infty}^2$, then we have
  \begin{align*}
    \|\tilde{\Vek y} -  \Vek y \|_{\infty}
    &\leq \biggl( \frac{2\sqrt{2} \, \pNormn{\Vek y}_\infty \, \pNormn{\Vek
        \psi}_\infty}{ \|\Vek y\|_{-\infty}^2 \, \pNormn{\Vek \psi}_{-\infty}^2 } +
    \frac{1}{2 \, \pNormn{\Vek \psi}_{-\infty} \, \sqrt{\|\Vek
        y\|_{-\infty}^2 \, \|\Vek \psi\|_{-\infty}^2-\epsilon}}\biggr)
    \, \epsilon
    \\ \shortintertext{and thus}
    \|\tilde{\Vek x} -  \Vek x \|_{\infty}
    &\leq \biggl( \frac{2\sqrt{2} \, \pNormn{\Vek y}_\infty \, \pNormn{\Vek
        \psi}_\infty}{ \|\Vek y\|_{-\infty}^2 \, \pNormn{\Vek \psi}_{-\infty}^2 } +
    \frac{1}{2 \, \pNormn{\Vek \psi}_{-\infty} \, \sqrt{\|\Vek
        y\|_{-\infty}^2 \, \|\Vek \psi\|_{-\infty}^2-\epsilon}}\biggr)
    \, \pNormn{\Mat S^{-1}}_1 \, \epsilon.
  \end{align*}
\end{proposition}

\begin{proof}
  The statement follows using the same technique as for
  \thref{prop:total-error}.  Notice however that in the last estimate
  $\absn{y_k}$ and $\absn{\psi_k}$ would not have to correspond to
  $\pNormn{\Vek y}_\infty$ and $\pNormn{\Vek \psi}_\infty$
  respectively since the phase is propagated from the coefficient
  $\tilde c_k \approx \bar y_k \psi_k$ largest in magnitude.
  Therefore the maximum norms do not cancel out.  For the second part,
  exploit $\Vek x = (\Mat S^*)^{-1}\Vek y$ and 
  $\| (\Mat S^{-1})^*\|_{\infty} = \| \Mat S^{-1} \|_1$.  \qed
\end{proof}

\paragraph{Multiple sampling vectors}

Finally, we would like to discuss the sensitivity of the phase
propagation in the setting of \thref{thm:sig-sys:sampl-set}, where we
exploit spatiotemporal measurements with respect to several sampling
vectors $\Vek\phi_i$.  Here we first recover the partial spectra
$\tilde{\Lambda}_i =\{ \tilde{\lambda}_k : k \in \supp \psi_i\}$ up to
global phase and winding direction, then identify the order within the
partial spectra, and afterwards align these to find the complete
spectrum of $\Mat A$ with one unified global phase and winding
direction.  In this process an extra error will appear in the phase of
eigenvalues because of the phase propagation between the partial
spectra. Fortunately, the amplitude of the eigenvalues is not
affected.

To demonstrate the issue in more detail, let us -- for the
moment -- consider two partial spectra $\tilde\Lambda_0$ and
$\tilde\Lambda_1$ and assume
\begin{equation*}
  \absn{\arg(\tilde{\lambda}_{i,k}) - \arg( \lambda_{k}) \Mod 2 \uppi}
  \le \rho
\end{equation*}
if $\lambda_{k}$ is covered by $\tilde\Lambda_i$.  For simplicity, we
assume that the winding directions are already aligned.  If we now
propagate the phase from $\tilde\Lambda_0$ over $\tilde\lambda_{0,k}$
and $\tilde\lambda_{1,k}$ to $\tilde\Lambda_1$, then the phases in
$\tilde\Lambda_1$ have to be shifted by
$\arg(\tilde{\lambda}_{0,k})-\arg(\tilde{\lambda}_{1,k})$.  Since the
phase of $\tilde\lambda_{1,k}$ is already defective, the error within
$\tilde{\Lambda}_1$ may accumulate at most to $2\rho$.  If we want to
align the global phase of the entire spectrum, we may take the element
with the largest magnitude in $\tilde\Lambda_0$, look for the shortest
path over the partial spectra $\tilde\Lambda_i$ to $\tilde \lambda_j$,
and propagate the phase along this path.  The error of
$\arg(\tilde\lambda_j)$ may then accumulate at most to $[1+2(M-1)]\rho$, where
$M$ is the number of the employed spectra $\Lambda_i$.  A schematic
example of this procedure is shown in \autoref{fig:patch-phase-pro}.
For the phase of the transformed signal $\Vek y$, we may apply the
same procedure.

\begin{figure}
  \centering
  \includegraphics{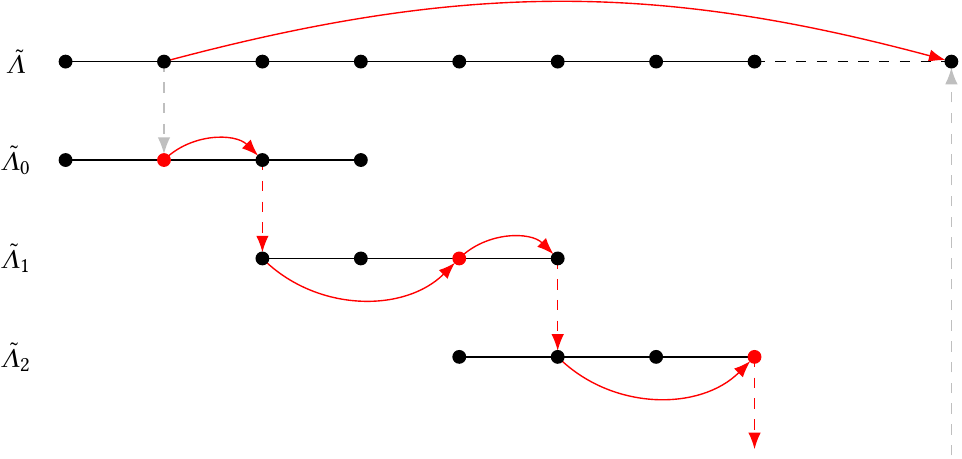}
  \caption{Schematic example for the propagation of the phase from
    some starting element over some path to another element.  The
    elements of $\tilde\Lambda_i$ with the largest magnitude are
    marked in red.  In each partial spectra, the phase error get worse
    by $2\rho$ at the most.}
  \label{fig:patch-phase-pro}
\end{figure}

\section{Numerical examples}
\label{sec:numerical-examples}

The constructive proofs of the uniqueness guarantees for phase
retrieval and system identification can immediately be implemented to
obtain numerical algorithms.  Because of the sensitivity of Prony's
method as corner stone of the proofs, these methods will however be
vulnerable to noise.  Nevertheless, we provide some small numerical
examples to accompany the theoretical results and to show that
simultaneous identification of system and signal is possible in
principle.  All numerical experiments have been implemented
in Julia\footnote{The Julia Programming Language -- Version 1.4.2
  (\url{https://docs.julialang.org})}.

\begin{example}[Prony's method]
  \label{ex:prony}
  First, we apply the approximated Prony method in
  \thref{alg:app-prony} to the complex setting.  For this, we
  generate exponential sums \eqref{eq:exp-sum} by choosing the
  coefficients and bases from a ring in the complex plane.  More
  precisely, the absolute values are drawn with respect to the uniform
  distributions $\abs{\eta_k} \sim \mathcal U ([\nicefrac18, 1])$ and
  $\abs{\beta_k} \sim \mathcal U ([\nicefrac12, 1])$ and the phases
  form $\mathcal U((-\uppi, \uppi])$ independently.  The mean maximal
  reconstruction errors for different numbers of addends $K$ and
  numbers of samples $L$.  The results over 5\,000 reconstructions are
  recorded in \autoref{tab:err-bases:noise-free} and
  \ref{tab:err-coeff:noise-free}.  For a small number of addends, the
  parameter are identified fairly well.  Increasing the number of
  addends however leads to a significant loss of accuracy.  To some
  degree, this may be compensated by employing more samples.  We
  repeat this experiment with small additive noise
  $\abs{e_k} \sim U([0, 10^{-10}])$ and
  $\arg(e_k) \sim \mathcal U((-\uppi, \uppi])$, see
  \autoref{tab:err-bases:noisy} and \ref{tab:err-coeff:noisy}.   \qed
\end{example}

\begin{table}[p]
  \centering\scriptsize
  \begin{tabular}{|c|cccccc|}
    \hline
    & \multicolumn{6}{c|}{Number of samples $L$}
    \\
    $K$
    & $2K +1$
    & $3K +1$
    & $4K +1$
    & $5K +1$
    & $8K +1$
    & $10K +1$
    \\ \hline
    5
    & $7.380 \cdot 10^{-12}$
    & $1.317 \cdot 10^{-12}$
    & $1.446 \cdot 10^{-12}$
    & $9.931 \cdot 10^{-13}$
    & $2.162 \cdot 10^{-13}$
    & $4.886 \cdot 10^{-13}$
    \\
    10
    & $1.212 \cdot 10^{-7}$
    & $1.822 \cdot 10^{-7}$
    & $4.340 \cdot 10^{-8}$
    & $1.526 \cdot 10^{-8}$
    & $4.081 \cdot 10^{-9}$
    & $5.598 \cdot 10^{-9}$
    \\
    15
    & $1.286 \cdot 10^{-3}$
    & $2.475 \cdot 10^{-4}$
    & $1.646 \cdot 10^{-6}$
    & $3.558 \cdot 10^{-5}$
    & $2.163 \cdot 10^{-6}$
    & $2.460 \cdot 10^{-6}$
    \\
    20
    & $1.503 \cdot 10^{-2}$
    & $5.406 \cdot 10^{-4}$
    & $7.727 \cdot 10^{-4}$
    & $3.951 \cdot 10^{-4}$
    & $2.998 \cdot 10^{-4}$
    & $2.325 \cdot 10^{-4}$
    \\ \hline
  \end{tabular}
  \caption{The mean of the reconstruction error $\pNormn{\Vek \beta -
      \tilde{\Vek \beta}}_\infty$ over 5\,000 experiments for
    different numbers of addends $K$ and samples $L$ in the noise-free
  setting, see Example~\ref{ex:prony}.}
  \label{tab:err-bases:noise-free}
\end{table}

\begin{table}[p]
  \centering\scriptsize
  \begin{tabular}{|c|cccccc|}
    \hline
    & \multicolumn{6}{c|}{Number of samples $L$}
    \\
    $K$
    & $2K +1$
    & $3K +1$
    & $4K +1$
    & $5K +1$
    & $8K +1$
    & $10K +1$
    \\ \hline
    5
    & $1.298 \cdot 10^{-10}$
    & $5.154 \cdot 10^{-11}$
    & $4.864 \cdot 10^{-11}$
    & $3.644 \cdot 10^{-11}$
    & $6.745 \cdot 10^{-12}$
    & $1.988 \cdot 10^{-11}$
    \\
    10
    & $3.517 \cdot 10^{-6}$
    & $6.538 \cdot 10^{-6}$
    & $6.285 \cdot 10^{-6}$
    & $5.281 \cdot 10^{-7}$
    & $1.483 \cdot 10^{-7}$
    & $2.763 \cdot 10^{-7}$
    \\
    15
    &  $2.194 \cdot 10^{-3}$
    & $1.403 \cdot 10^{-4}$
    & $1.193 \cdot 10^{-4}$
    & $2.193 \cdot 10^{-4}$
    & $6.814 \cdot 10^{-5}$
    & $6.406 \cdot 10^{-5}$
    \\
    20
    &  $1.860 \cdot 10^{-2}$
    & $2.040 \cdot 10^{-3}$
    & $2.445 \cdot 10^{-3}$
    & $1.503 \cdot 10^{-3}$
    & $2.021 \cdot 10^{-3}$
    & $1.405 \cdot 10^{-3}$
    \\ \hline
  \end{tabular}
  \caption{The mean of the reconstruction error $\pNormn{\Vek \eta -
      \tilde{\Vek \eta}}_\infty$ over 5\,000 experiments for
    different numbers of addends $K$ and samples $L$ in the noise-free
  setting, see Example~\ref{ex:prony}.}
  \label{tab:err-coeff:noise-free}
\end{table}

\begin{table}[p]
  \centering\scriptsize
  \begin{tabular}{|c|cccccc|}
    \hline
    & \multicolumn{6}{c|}{Number of samples $L$}
    \\
    $K$
    & $2K +1$
    & $3K +1$
    & $4K +1$
    & $5K +1$
    & $8K +1$
    & $10K +1$
    \\ \hline
    5
    & $1.481 \cdot 10^{-5}$
    & $7.640 \cdot 10^{-6}$
    & $5.188 \cdot 10^{-7}$
    & $1.942 \cdot 10^{-7}$
    & $3.555 \cdot 10^{-7}$
    & $3.335 \cdot 10^{-7}$
    \\
    10
    & $1.580 \cdot 10^{-2}$
    & $4.646 \cdot 10^{-3}$
    & $3.571 \cdot 10^{-3}$
    & $3.210 \cdot 10^{-3}$
    & $3.442 \cdot 10^{-3}$
    & $3.413 \cdot 10^{-3}$
    \\
    15
    & $9.528 \cdot 10^{-2}$
    & $2.016 \cdot 10^{-2}$
    & $1.719 \cdot 10^{-2}$
    & $1.570 \cdot 10^{-2}$
    & $1.685 \cdot 10^{-2}$
    & $1.290 \cdot 10^{-2}$
    \\
    20
    & $2.741 \cdot 10^{-1}$
    & $9.357 \cdot 10^{-2}$
    & $8.451 \cdot 10^{-2}$
    & $7.909 \cdot 10^{-2}$
    & $8.243 \cdot 10^{-2}$
    & $8.477 \cdot 10^{-2}$
    \\ \hline
  \end{tabular}
  \caption{The mean of the reconstruction error
    $\pNormn{\Vek \beta - \tilde{\Vek \beta}}_\infty$ over 5\,000
    experiments for different numbers of addends $K$ and samples $L$
    in the noisy setting $\abs{e_k} \sim U([0, 10^{-10}])$ and
    $\arg(e_k) \sim \mathcal U((-\uppi, \uppi])$, see
    Example~\ref{ex:prony}.}
  \label{tab:err-bases:noisy}
\end{table}

\begin{table}[p]
  \centering\scriptsize
  \begin{tabular}{|c|cccccc|}
    \hline
    & \multicolumn{6}{c|}{Number of samples $L$}
    \\
    $K$
    & $2K +1$
    & $3K +1$
    & $4K +1$
    & $5K +1$
    & $8K +1$
    & $10K +1$
    \\ \hline
    5
    & $2.680 \cdot 10^{-4}$
    & $2.120 \cdot 10^{-4}$
    & $1.215 \cdot 10^{-5}$
    & $3.500 \cdot 10^{-6}$
    & $1.419 \cdot 10^{-5}$
    & $6.576 \cdot 10^{-6}$
    \\
    10
    & $1.580 \cdot 10^{-2}$
    & $4.646 \cdot 10^{-3}$
    & $3.571 \cdot 10^{-3}$
    & $3.210 \cdot 10^{-3}$
    & $3.442 \cdot 10^{-3}$
    & $3.413 \cdot 10^{-3}$
    \\
    15
    & $9.304 \cdot 10^{-2}$
    & $3.209 \cdot 10^{-2}$
    & $3.271 \cdot 10^{-2}$
    & $2.968 \cdot 10^{-2}$
    & $3.093 \cdot 10^{-2}$
    & $2.804 \cdot 10^{-2}$
    \\
    20
    & $2.256 \cdot 10^{-1}$
    & $1.224 \cdot 10^{-1}$
    & $1.178 \cdot 10^{-1}$
    & $1.135 \cdot 10^{-1}$
    & $1.184 \cdot 10^{-1}$
    & $1.230 \cdot 10^{-1}$
    \\ \hline
  \end{tabular}
  \caption{The mean of the reconstruction error
    $\pNormn{\Vek \eta - \tilde{\Vek \eta}}_\infty$ over 5\,000
    experiments for different numbers of addends $K$ and samples $L$
    in the noisy setting $\abs{e_k} \sim U([0, 10^{-10}])$ and
    $\arg(e_k) \sim \mathcal U((-\uppi, \uppi])$, see
    Example~\ref{ex:prony}.}
  \label{tab:err-coeff:noisy}
\end{table}

\begin{table}
  \centering\scriptsize
  \begin{tabular}{|c|cccccc|}
    \hline
    & \multicolumn{6}{c|}{Index $k$ in time domain}
    \\
    & 0
    & 1
    & 2
    & 3
    & 4
    & 5
    \\ \hline
    $x$
    & -0.806\,494\,570\,188
    & 0.697\,047\,937\,358
    & 0.475\,340\,169\,748
    & -0.868\,496\,176\,947
    & -0.373\,776\,219\,367
    & 0.573\,125\,494\,692
    \\
    $\phi_1$
    & 0.299\,100\,737\,288
    & -0.067\,652\,854\,127
    & 0.223\,548\,074\,051
    & -0.419\,039\,372\,471
    & 0.398\,336\,559\,020
    & 0.439\,827\,094\,742
    \\
    $\phi_2$
    & -0.222\,947\,251\,005
    & 0.185\,111\,331\,800
    & 0.508\,076\,580\,285
    & -0.024\,006\,689\,074
    & 0.491\,191\,477\,978
    & -0.360\,304\,943\,116
    \\ \hline
  \end{tabular}
  \caption{The randomly generated unknown signal $\Vek x$ and the
    known measurement vectors $\Vek \phi_1$, $\Vek \phi_2$ in
    Example~\ref{ex:sym-phase-ret} satisfying the assumptions of
    Theorem~\ref{thm:sys-sig:real-case}.}
  \label{tab:sym-phase-ret:sig}
\end{table}

\begin{example}[Simultaneous signal \& system identification]
  \label{ex:sym-phase-ret}
  In this numerical example, we consider the recovery of real-valued
  signals and convolution kernels as discussed in
  \autoref{sec:simult-phase-system}.  The true, unknown kernel
  $\Vek a \in \BR^6$ is here chosen as
  \begin{equation*}
    \hat{\Vek a}
    \coloneqq
    \big( \cos(2 k) \bigr)_{k=-3}^2,
  \end{equation*}
  where the indices are considered modulo 6.  Besides the
  strictly, symmetrically decreasing kernel, the unknown signal
  $\Vek x \in \BR^6$ and the known measurement vectors
  $\Vek \phi_1, \Vek \phi_2 \in \BR^6$ have been randomly generated
  such that the requirements for the reconstruction are fulfilled,
  \ie\ $\Vek \phi_1$ and $\Vek \phi_2$ are pointwise independent in
  the frequency domain, and the assumption
  $\Re[\bar{\hat x}_k \hat \phi_{i,k}] \ne 0$ is satisfied for
  $k = 0, \dots, 5$, $i = 1,2$.  For reproducibility, the employed
  signals are shown in \autoref{tab:sym-phase-ret:sig}.  Choosing
  $L \coloneqq 4 d^2 + 1 = 145$ to encounter the numerical sensitivity
  of Prony's method, we now apply the procedure in the constructive
  proof of \thref{thm:sys-sig:real-case}.  The reconstructions
  $\tilde{\Vek a}$ and $\tilde{\Vek x}$ of the true signals $\Vek a$
  and $\Vek x$ are shown in \autoref{fig:sym-phase-ret:sig}.  Aligning
  the overall sign of $\Vek x$ and $\tilde{\Vek x}$, we are able to
  recover the unknown signals up to an error of
  $\pNormn{\hat{\Vek a} - \hat{\tilde{\Vek a}}}_\infty = 8.650 \cdot
  10^{-5}$ and
  $\pNormn{\Vek x - \tilde{\Vek x}}_\infty = 1.141 \cdot 10^{-3}$.
  The theoretical procedure behind \thref{thm:sys-sig:real-case} thus
  allows the simultaneous recovery of signal and kernel numerically at
  least for small instances.  \qed
\end{example}

\begin{figure}
  \centering
  \subfloat[Convolution kernel $\hat{\Vek a}$.]{
    \includegraphics{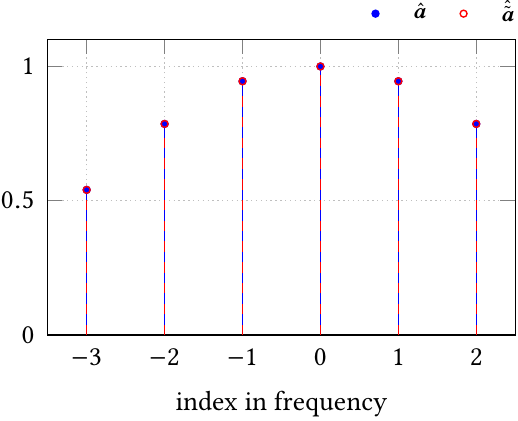}
  }
  \qquad
  \subfloat[Unknown signal $\Vek x$.]{
    \includegraphics{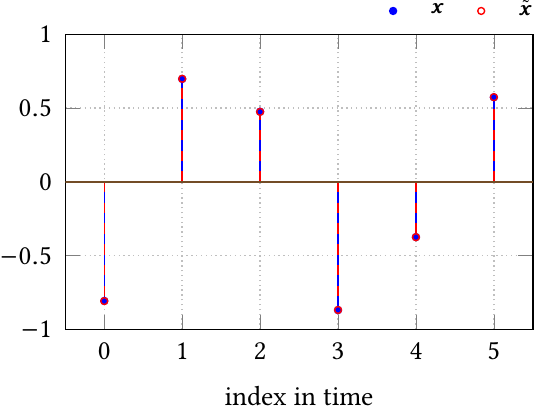}
  }
  \caption{The true and reconstructed signal and kernel in
    Example~\ref{ex:sym-phase-ret} by applying the procedure provided in
    Theorem~\ref{thm:sys-sig:real-case}.}
  \label{fig:sym-phase-ret:sig}
\end{figure}

\begin{example}[Multiple sampling vectors]
  \label{ex:mult-samp-vec}
  Finally, we consider the identification of\linebreak complex-valued
  signals and convolution kernels, \ie\
  $\Mat A \coloneqq \Circ \Vek a$, using multiple sampling vectors.
  For the experiment, the true but unknown signal $\Vek x \in \BC^{50}$ and
  kernel $\Vek a \in \BC^{50}$ have been randomly generated such that
  $\Vek x$ has a non-vanishing Fourier transform and $\Vek a$ is
  absolutely collision-free, see \autoref{fig:mult-samp-vec}.
  Further, we generate 47 sampling vectors $\Vek \phi_i \in \BC^{50}$
  such that $\supp \hat{\Vek \phi_i} = \{ i, \dots, i+3\}$.  Since the
  support of two consecutive sampling vectors is shifted by one, the
  generated sampling vectors allow index separation \eqref{eq:ind-sep}
  and phase propagation \eqref{eq:phase-prop}.  Additionally, we
  ensure that the winding direction determination property
  \eqref{eq:wind-dir-det} is satisfied for
  $i_1 = 0, i_1 = 1, k_1 = 1, k_2 = 2$.  Further, we employ for each
  sampling vector 65 samples, which is around twice the minimal
  required number to apply Prony's method.  Next, we apply the
  construction behind the proof of \thref{thm:sig-sys:sampl-set} line
  by line, where the procedure in the proof of
  \thref{thm:phase-eig-values} is used to identify the partial
  spectrum of $\Vek a$ with respect to $\Vek \phi_i$.  The recovered signal
  $\tilde{\Vek x}$ and kernel $\tilde{\Vek a}$ are shown in
  \autoref{fig:mult-samp-vec}.  Aligning the phase of the true and
  recovered vectors at the first component, we here observe the
  reconstruction errors
  $\pNormn{\hat{\Vek a} - \hat{\tilde{\Vek a}}}_\infty = 1.897 \cdot
  10^{-3}$ and
  $\pNormn{\hat{\Vek x} - \hat{\tilde{\Vek x}}}_\infty = 1.563 \cdot
  10^{-4}$.  As shown in this example, the techniques behind the
  theoretical proofs may be applied to recover signal and kernel from
  noise-free samples.  \qed
\end{example}

\begin{figure}[t]
  \centering
  \subfloat[Magnitude of the kernel in frequency.]{
    \includegraphics{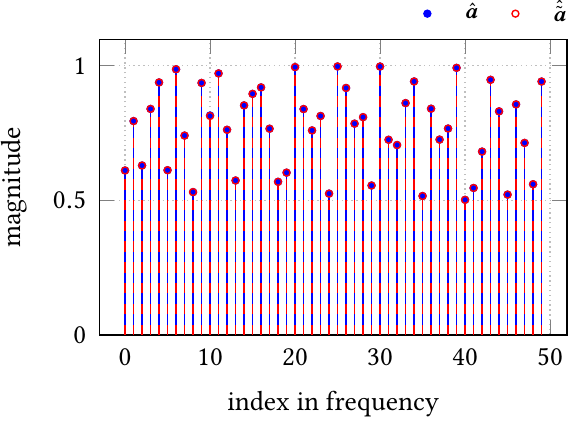}
  }
  \qquad
  \subfloat[Phase of the kernel in frequency.]{
    \includegraphics{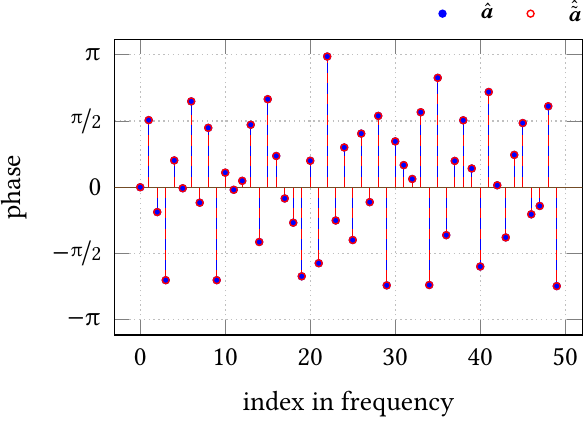}
  }
  \newline
  \subfloat[Magnitude of the signal in frequency.]{
    \includegraphics{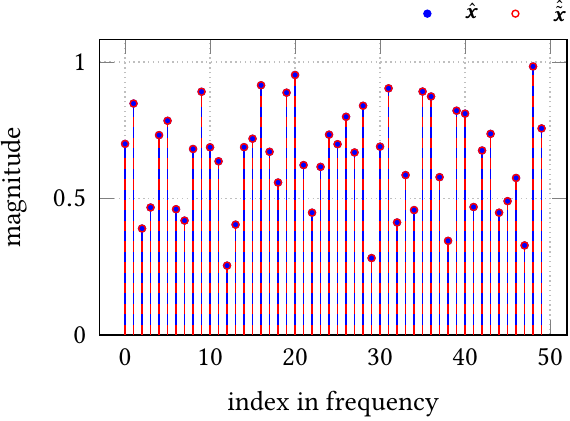}
  }
  \qquad
  \subfloat[Phase of the signal in frequency.]{
    \includegraphics{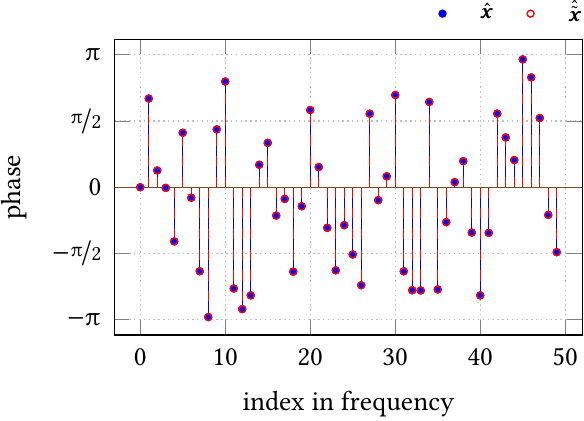}
  }
  \caption{The true and reconstructed signal and kernel in
    Example~\ref{ex:mult-samp-vec} by applying the procedure behind
    Theorem~\ref{thm:sig-sys:sampl-set}.}
  \label{fig:mult-samp-vec}
\end{figure}

\section{Conclusion}
\label{sec:conclusion}

Phase retrieval in dynamical sampling is a novel research direction
occurring a few years ago.  As for most phase retrieval problems, the
main issue is the ill-posedness especially emerging in the
non-uniqueness of the solution.  Besides the phase retrieval of the
unknown signal, we additionally identify the unknown involved operator
from a certain operator class.  We have shown that both -- phase
retrieval and system identification -- is in principle simultaneously
possible if the spectrum of the operator is (absolutely)
collision-free.  The employed conditions to ensure the uniqueness of
the combined phase and system identification hold for almost all
signals, spectra, and measurement vectors.  Our work horse has been
the approximate Prony method for complex exponential sums.  As a
consequence, all proofs are constructive and give explicit analytic
reconstruction methods.  Unfortunately, Prony's method is notorious
for its instability.  We have studied the sensitivity in more details
yielding error bounds that are interesting by themselves outside the
context of dynamical sampling.  The recovery error of phase and system
here centrally depends on the well-separation of the pairwise products
of the spectrum and how far the involved entities are away from zero.
Especially for high-dimensional instances the well-separation gets
worse and worse since the pairwise products start to cluster; so the
analytic reconstructions can only be applied to small instances or a
series of specially constructed sampling vectors numerically.  The
main contributions of this paper are the theoretical uniqueness
guarantees, where the question of a practical recovery methods remains
open for further research.  In particular for phase retrieval, it
would be interesting to adapt Prony's method to the occurring
quadratic structure or to replace it by a more suitable method.

\bibliographystyle{abbrv}
{\footnotesize \bibliography{literature}}

\end{document}